\documentclass[11pt]{amsart}

\makeatletter
\def\paragraph{\@startsection{paragraph}{4}%
  \z@\z@{-\fontdimen2\font}%
  {\normalfont\bfseries}}
\makeatother

\newcommand{\para}[1]{\medskip\noindent\textbf{#1.}}

\usepackage{amsmath}
\usepackage{amssymb}
\usepackage{amsthm}
\usepackage{enumitem}
\usepackage{microtype}
\usepackage{array}
\usepackage{dcolumn}
\usepackage[hmargin=1.3in, vmargin=1.30in]{geometry}
\usepackage{float}
\usepackage{graphicx}
\usepackage{tikz-cd}
\usepackage[maxbibnames=99,style=alphabetic]{biblatex}
\usepackage[hidelinks]{hyperref}
\usepackage{nicematrix}
\usepackage{marginnote}
\usepackage{colortbl}

\renewcommand \v{\vert}

\renewcommand \Im{\text{Im}}
\newcommand \Q{\mathbb{Q}}
\newcommand \R{\mathbb{R}}

\newcommand \Z{\mathbb{Z}}
\newcommand \C{\mathbb{C}}

\newcommand \ep{\varepsilon}
\newcommand \la{\langle}
\newcommand \ra{\rangle}
\newcommand \del{\partial}

\DeclareMathOperator \Ker{Ker}
\DeclareMathOperator \Int{Int}

\DeclareMathOperator \Aut{Aut}
\DeclareMathOperator \Mod{Mod}
\DeclareMathOperator \PMod{PMod}
\DeclareMathOperator \GL{GL}
\DeclareMathOperator \SL{SL}
\DeclareMathOperator \PSL{PSL}
\DeclareMathOperator \Sp{Sp}
\DeclareMathOperator \SU{SU}

\DeclareMathOperator \U{U}

\DeclareMathOperator \Hom{Hom}

\DeclareMathOperator \Stab{Stab}

\DeclareMathOperator \PConf{PConf}
\DeclareMathOperator \PB{PB}

\DeclareMathOperator \LMod{LMod}
\DeclareMathOperator \Homeo{Homeo}
\DeclareMathOperator \Irr{Irr}
\DeclareMathOperator \End{End}
\newcommand \G{\mathcal{G}}

\theoremstyle{definition}
\newtheorem{definition}{Definition}[section]

\theoremstyle{plain}
\newtheorem{theorem}[definition]{Theorem}

\newtheorem{lemma}[definition]{Lemma}

\newtheorem{question}[definition]{Question}
\theoremstyle{remark}

\numberwithin{equation}{section}

\title{Homological representations of low genus mapping class groups}
\author{Trent Lucas}

\begin{document}
\begin{abstract}
Given a finite group $G$ acting on a surface $S$, the centralizer of $G$ in the mapping class group $\Mod(S)$ has a natural representation given by its action on the homology $H_1(S;\Q)$.
We  consider the question of whether this representation has arithmetic image.  Several authors have given positive and negative answers to this question.  We give a complete answer when $S$ has genus at most 3.
\end{abstract}
\maketitle

\section{Introduction}
The mapping class group $\Mod(S)$ of a closed surface $S$ acts on the homology $H_1(S;\Q)$ and preserves the intersection form, yielding the classical \emph{symplectic representation} $\Mod(S) \rightarrow \Sp(H_1(S;\Q))$.  In this paper, we study an equivariant refinement of the symplectic representation.  Namely, given an action of a finite group $G$ on a surface $S$, the symplectic representation restricts to a map on centralizers $\Phi:\Mod(S)^G \rightarrow \Sp(H_1(S;\Q))^G$.  It is well known that the image of the symplectic representation is $\Sp(H_1(S;\Z)) \cong \Sp(2g,\Z)$.  In particular, the image is an \emph{arithmetic subgroup}, meaning it has finite index in the integer points of its Zariski closure.  In the equivariant case, we can ask whether the same is true: is the image of $\Phi$ an arithmetic subgroup?

This question can be refined as follows.  Observe that $H_1(S;\Q)$ splits into isotypic components $H_1(S;\Q)_V$ for each irreducible $\Q$-representation $V$ of $G$.  The group $\Sp(H_1(S;\Q))^G$ preserves this decomposition and hence splits as a product.  So, $\Phi$ induces representations
\begin{equation*}
\Phi_V:\Mod(S)^G \rightarrow \Aut_{\Q[G]}(H_1(S;\Q)_V, \hat{i}),
\end{equation*}
where $\hat{i}$ is the intersection form.  We now ask:

\begin{question}[The arithmeticity question]
For a given irreducible representation $V$, is the image of $\Phi_V$ an arithmetic subgroup?
\end{question}

If the action of $G$ is free and $S/G$ has sufficiently high genus, then for every $V$, the answer to the arithmeticity question is known to be ``yes" for abelian $G$ by Looijenga \cite{Loo1}, and ``yes" for many non-abelian cases by Grunewald, Larsen, Lubotzky, and Malestein \cite{GLLM}.  On the other hand, Deligne and Mostow \cite{DM} produced examples of non-free cyclic actions where the answer is ``no".  As a first step towards understanding precisely how often the image of $\Phi_V$ is arithmetic, we offer a complete answer for low genus cases.
\begin{theorem}\label{thm:main}
If $S$ has genus $g \leq 3$, then for every action of a finite group $G$ on $S$ and for every irreducible $\Q$-representation $V$ of $G$, the image of the map $\Phi_V:\Mod(S)^G \rightarrow \Aut_{\Q[G]}(H_1(S;\Q)_V, \hat{i})$ is an arithmetic subgroup.
\end{theorem}

In some cases, the image of $\Phi_V$ is finite and hence automatically arithmetic.  However, in many other cases we show that the image of $\Phi_V$ is in fact commensurable to the subgroup $\Aut_{\Z[G]}(H_1(S;\Z)_V, \hat{i})$.  Our main strategy to do so is to develop an algorithm for computing $\Phi_V(f)$ when $f \in \Mod(S)^G$ is a lift of any Dehn twist along the branched cover $S \rightarrow S/G$.

\para{Branched cover perspective}
We can apply Theorem \ref{thm:main} to obtain arithmetic representations of the mapping class group of the quotient.  That is, if we let $S^\circ$ denote the complement of the points with nontrivial $G$-stabilizers, then we can lift mapping classes to obtain a representation $\Gamma \rightarrow \Aut_{\Q[G]}(H_1(S;\Q)_V, \hat{i})$ on a finite index subgroup $\Gamma \subseteq \Mod(S^\circ/G)$.  The precise construction, following \cite{Loo1} and \cite{GLLM}, is detailed in Section 2.2. We note that if $S/G$ has genus $h \geq 3$, then Putman and Wieland \cite{PW} relate these representations to Ivanov's question of whether $\Mod(S^\circ/G)$ virtually surjects onto $\Z$.  However, $S/G$ has genus $h < 3$ in all cases of Theorem \ref{thm:main}.

\para{The target group}
The target group $\G_V := \Aut_{\Q[G]}(H_1(S;\Q)_V, \hat{i})$ can be identified as $\Aut_{A_V}(H_1(S;\Q)_V, \eta)$, where $A_V \subseteq \Q[G]$ is the corresponding simple factor of the Wedderburn decomposition and $\eta$ is a $\Q[G]$-valued skew-Hermitian form called the \emph{Reidemeister pairing} (see Section 2.2).  In all the interesting cases of Theorem \ref{thm:main}, we identify the group $\G_V$ explicitly as a symplectic group (with the important special case $\SL_2(\Q) = \Sp_2(\Q)$) or a unitary group.  More generally, as explained in \cite{GLLM}, the extension of scalars $\G_V(\C)$ will always be isomorphic to some $\mathrm{O}_k(\C)$, $\Sp_{2k}(\C)$, or $\GL_k(\C)$.

\para{Relation with prior work}
As mentioned above, the arithmeticity question has been answered in many cases.  Looijenga \cite{Loo1} proved that the answer is ``yes" when $G$ is abelian, $G$ acts freely, and $S/G$ has genus $h \geq 2$;  in fact, he found the images of the representations $\Phi_V$ explicitly.  For an arbitrary group $G$ acting freely, Grunewald, Larsen, Lubotzky, and Malestein \cite{GLLM} proved that if $G$ acts freely, and $S/G$ has genus $h \geq 3$, and the action extends to a ``redundant" action on a handlebody, then the answer to the arithmeticity question is ``yes."  More recently, Looijenga \cite{Loo2} gave an arithmeticity criterion for arbitrary actions where $S/G$ has genus $h \geq 2$.  There are also several results known for cyclic branched covers of the sphere.  For the hyperelliptic involution, A'Campo \cite{ACA} answered the question affirmatively and in fact found the image of $\Phi$ explicitly. Later, Deligne and Mostow \cite{DM} \cite{Mos} studied the monodromy of the family of curves
\begin{equation*}
y^m = (x-b_1)^{k_1} \cdots (x - b_n)^{k_n}
\end{equation*}
for distinct points $b_1, \ldots, b_n \in \C$ and $k_i,m \in \Z_{> 0}$.  As we explain in Section 3.2, the monodromy of this family coincides with the representations $\Phi_V$ for the cyclic group $C_m$.  They answer the arithmeticity question for certain values of $k_i$ and $m$ (``yes" in some cases, ``no" in others).  McMullen \cite{McM} answers the arithmeticity question affirmatively in some additional cases where $k_1 = \cdots = k_n$.  The monodromy of this family also coincides with the the Gassner representation of the pure braid group specialized at an $m$th root of unity (one obtains the Burau representation in the case $k_1 = \cdots = k_n = 1$), which was proved to have arithmetic image by Venkataramana \cite{Ven2, Ven1} when $n-1 \geq 2m$ and each $k_i$ is a unit mod $m$.

In this paper, we tackle several cases which are not covered by the above works, and where their methods do not apply.  The moduli space connections in \cite{DM} and \cite{McM} do not readily extend to non-cyclic $G$.  The proofs of arithmeticity in \cite{Ven2, Ven1}, \cite{GLLM}, and \cite{Loo2} use the fact that in a high rank Lie group, one can generate an arithmetic subgroup with ``enough" unipotent elements (c.f. \cite[Theorem~7]{Ven1}).  However, in almost every case we study, the group $\G_V(\R)$ has rank 1.  Moreover, the most natural way to find unipotent elements in the image of $\Phi_V$ is to lift Dehn twists around non-separating curves (e.g.\ as in \cite{Loo1} and \cite{Loo2}), but in most cases we consider $S/G$ has genus $h=0$, and hence there are no non-separating curves.

We also note that most of the low genus actions are not free. One additional challenge of working with non-free actions is that the character of $H_1(S;\Q)$ as a $G$-representation depends on the specific action, as opposed to the free case where it depends only on $G$.  Moreover, when $G$ acts freely or $G$ is abelian (i.e.\ in most of the works above), the isotypic component $H_1(S;\Q)_V$ is always a free module over the corresponding simple factor $A_V \subseteq \Q[G]$, but it need not be free in general (see Section 3.3), which further complicates the situation.

Several authors have studied aspects of homological representations beyond the arithmeticity question, including applications to 3-manifold topology; see e.g.\ \cite{Hadari_homological_eigenvalues}, \cite{Koberda_aymptotic_linearity}, \cite{Koberda_alexander_varieties}, \cite{Liu_homological_spectral_radii}, or \cite{Sun_homological_spectral_radius}.

\para{About the proof}
The arithmeticity of the image of $\Phi_V$ only depends on the action of $G$ up to conjugation in $\mathrm{Homeo}(S)$ and automorphisms of $G$.  Up to this equivalence, there are only finitely many actions of finite groups on surfaces of genus $g \leq 3$.  The arithmeticity question is trivial in the cases $g=0$ and $g=1$, and so it suffices to examine the finitely-many actions on genus 2 and 3 surfaces.  In \cite{Bro}, Broughton constructs a complete list of such actions.  We summarize the 35 cases where the arithmeticity question is nontrivial in Table \ref{table:main}.

Twenty of the cases are immediate due to having a finite image, factoring through smaller genus cases, or following from prior work on cyclic actions.  In 3 cases, we show by hand that $\Im(\Phi_V)$ is arithmetic by understanding the action geometrically. In the 12 remaining cases, the image of $\Phi_V$ is (conjugate to) a subgroup of $\SL_2(\Z)$.  Given a set of elements in $\SL_2(\Z)$, there is an algorithm to determine whether they generate a finite index subgroup.  Thus, as long as we can compute $\Phi_V(f)$ for a set of $f \in \Mod(S)^G$, it is easy to check if these matrices are enough to generate an arithmetic subgroup.  We can find elements of $\Mod(S)^G$ by lifting elements of $\Mod(S^\circ/G)$, but in general it is not easy to compute the action of such a lift on $H_1(S;\Q)$.  There is one class of elements in $\Mod(S)^G$ which admit simple geometric descriptions, namely lifts of powers of Dehn twists, which we call \emph{partial rotations}.  Surprisingly, in all 12 cases, we find that partial rotations are in fact enough to generate an arithmetic group. 

To show that one can generate an arithmetic subgroup using partial rotations, we adapt the geometric description into an algorithm to compute their action on homology.  To build this algorithm, we equip $S$ with an explicit $G$-compatible cell structure; the partial rotation is not quite a cellular map, but we can directly define a map on the cellular chain group whose induced map on $H_1(S;\Q)$ agrees with the action of the partial rotation.    We carry out the computations by implementing the algorithm in SageMath.  The code is available in a repository at \cite{low-genus-actions}.

\para{Further questions}
As suggested above, in the final 12 cases, there is no reason to expect a priori that partial rotations generate a finite index subgroup of $\G_V(\Z)$.  A natural question is whether $\Mod(S)^G$ is generated by partial rotations in these cases, or equivalently, whether the liftable subgroup of $\Mod(S^\circ/G)$ is generated by liftable Dehn twists.  More generally, it would be of interest to find generating sets of $\Mod(S)^G$ whose generators admit simple geometric descriptions.

Additionally, the key property that makes our final 12 cases tractable is that these representations land in $\SL_2(\Q)$.  One might ask if there are any actions where $\G_V \cong \SL_2(\Q)$ but $\Im(\Phi_V)$ is not commensurable to $\SL_2(\Z)$.  Note that there are infinitely many actions where $\G_V \cong \SL_2(\Q)$ for a faithful irreducible representation $V$.  For example, if we let $F_3= \la x_1, x_2, x_3 \ra$ be the free group on 3 letters (i.e.\ the fundamental group of the four times puncture sphere), then the surjection $F_3 \rightarrow S_n$ given by $x_1 \mapsto (12)$, $x_2 \mapsto (12)$, and $x_3 \mapsto (12 \cdots n)$ gives an action of $G = S_n$ on a surface $S$ such that $S/G$ is a sphere with 4 branch points.  For $n \geq 5$, every irreducible representation of $G$ is faithful except for the trivial and sign representations, and each simple factor of $\Q[G]$ is isomorphic to some $M_k(\Q)$.  Thus the character formula in Theorem \ref{thm:H1asrep} tells us that $H_1(S;\Q)$ contains two copies of some faithful irreducible representation $V$.  By the arguments in Section 3.3, $\G_V \cong \SL_2(\Q)$.  We note that $S$ has large genus in these examples, and the current implementation of our algorithm is not fast enough to handle such cases.

\para{Outline}
The paper is organized as follows.  In Section 2, we establish the necessary background for the proof of Theorem \ref{thm:main}.  Namely, we recall the structure of $H_1(S;\Q)$ as a $\Q[G]$-module, define the Reidemeister pairing $\eta$, and carefully define the target group $\G_V$.  In Section 3, we begin the proof of Theorem \ref{thm:main}.  We first summarize the cases and explain which ones follow from simple reductions or prior work.  For the remaining cases, we identify the target group $\G_V$ explicitly, and tackle some cases ``by hand".  In Section 4, we finish the proof of Theorem \ref{thm:main} by studying the cases where $\Im(\Phi_V) \subseteq \SL_2(\Z)$ (up to conjugation and commensurability).  We define partial rotations, develop an algorithm to compute their action on homology, and explain how to determine if the resulting subgroup of $\SL_2(\Z)$ has finite index.  We record the results of our computations in Table \ref{table:results}.

\para{Acknowledgements}
We thank Bena Tshishiku for introducing us to this problem and for countless helpful conversations throughout this work.  We also thank Ethan Dlugie,  Justin Malestein, and Nick Salter for helpful conversations on the topics surrounding this work, and we thank Benson Farb and Dan Margalit for helpful comments on an earlier version of this paper.

%-------------------------------------------------------------------------------------------------------------------------------------------------------------------

\section{Equivariant mapping classes and homology}

In this section, we recall some facts about the representation $\Mod(S)^G \rightarrow \Sp(H_1(S))^G$ to establish the necessary background for the proof of Theorem \ref{thm:main}.  In particular, the results in this section are all known by prior work, namely \cite{GLLM}, \cite{KS}, and the references therein.  In Section 2.1, we recall the structure of $H_1(S;\Q)$ as a $G$-representation, and in Section 2.2, we examine the target algebraic group $\G_V$.

\subsection{Homology as a representation}

Fix a finite group $G$, a closed surface $S$, and an action of $G$ on $S$.  Let $S^\circ$ denote the surface obtained by removing the points of $S$ with nontrivial $G$-stabilizers.  Then the quotient $S^\circ \rightarrow S^\circ/G$ is a covering map, and it extends to a branched cover $S \rightarrow S/G$.  The cover (and hence the action) is completely described by the monodromy homomorphism $\varphi:\pi_1(S^\circ/G) \rightarrow G$.  Suppose $S/G$ has genus $h$ and $n$ branch points $x_1, \ldots, x_n$.  Then we fix a presentation
\begin{equation*}
\pi_1(S^\circ/G) = \la c_1, d_1, \ldots, c_h, d_h, x_1, \ldots, x_n \mid [c_1,d_1] \cdots [c_h,d_h]x_1 \cdots x_n = 1\ra.
\end{equation*}
The monodromy homomorphism $\varphi$ is thus given by a $(2h+n)$-tuple of elements of $G$ which satisfies the relation of this presentation.

\para{Character formula} 
First, we present a formula for $H_1(S;\Q)$ as a $G$-representation.  Let $\Q[G]$ denote the rational group ring of $G$.  For every subgroup $H$ of $G$, the action of $G$ on the set of cosets $G/H$ gives rise to a permutation representation of $G$, which we denote by $\Q[G/H]$.

\begin{theorem}[Chevalley-Weil, Gasch\"utz, Koberda-Silberstein]\label{thm:H1asrep}
As a $\Q[G]$-module,
\begin{equation*}
H_1(S;\Q) \cong \frac{\Q^2 \oplus \Q[G]^{2h+n-2}}{\bigoplus_{i=1}^n \Q[G/\la \varphi(x_i) \ra]}
\end{equation*}
where $\Q$ denotes the trivial $\Q[G]$-module.
\end{theorem}

Note that $n=0$ if and only if $G$ acts freely on $S$, i.e.\ $S \rightarrow S/G$ is an unbranched cover.  In this case, Theorem \ref{thm:H1asrep} reduces to the classic formula of Chevalley-Weil \cite{CW}; see also \cite{GLLM} for a topological proof.  If $n > 0$, then the isomorphism in Theorem \ref{thm:H1asrep} is realized by the natural surjection $H_1(S^\circ;\Q) \rightarrow H_1(S; \Q)$; in particular, the result follows from two facts:
\begin{enumerate}
\item As $\Q[G]$-modules, $H_1(S^\circ;\Q) \cong \Q \oplus \Q[G]^{2h+n-2}$.
\item As $\Q[G]$-modules, $\Ker(H_1(S^\circ;\Q) \rightarrow H_1(S;\Q)) \cong \left(\oplus_{i=1}^n \Q[G/\la \varphi(x_i) \ra]\right)/\Q$.
\end{enumerate}
The first fact is due to Gasch\"utz \cite{Gas}; again see \cite{GLLM} for a topological proof.  The second fact is due to Koberda-Silberstein \cite{KS}; the idea is as follows.  Observe that $\Ker(H_1(S^\circ;\Q) \rightarrow H_1(S;\Q))$ is spanned by one small loop around each puncture, with the single relation that the sum of all these loops is trivial.  The action of $G$ on these loops is precisely the permutation action of $G$ on the punctures.  The formula then follows from the fact that the action of $G$ on the fiber over the branch point $x_i$ is the permutation representation $\Q[G/\la \varphi(x_i) \ra]$, and $G$ acts trivially on sum of the punctures.

\para{Isotypic components}
We can decompose $H_1(S;\Q)$ into its \emph{isotypic components}.  Namely, let $\Irr(G)$ denote the set of irreducible $\Q$-representations of $G$, and for every $V \in \Irr(G)$, let $m_V$ denote the multiplicity of $V$ in $H_1(S;\Q)$.  Then
\begin{equation*}
H_1(S;\Q) = \bigoplus_{V \in \Irr(G)} H_1(S;\Q)_V,
\end{equation*}
where $H_1(S;\Q)_V \cong V^{m_V}$.  Every $\Q[G]$-linear automorphism of $H_1(S;\Q)$ preserves this decomposition.

Moreover, recall the \emph{Wedderburn decomposition}
\begin{equation*}
\Q[G] \cong \prod_{V \in \Irr(G)} M_{n_V}(D_V),
\end{equation*}
where $D_V$ is the division algebra $\End_{\Q[G]}(V)$ and $n_V = \dim_{D_V}(V)$.  For each $V \in \Irr(G)$, the action of $\Q[G]$ on $V$ factors through $M_{n_V}(D_V)$.  We call $A_V := M_{n_V}(D_V)$ the \emph{simple factor} of $\Q[G]$ corresponding to $V$.  Each simple factor has a corresponding \emph{idempotent} $e_V  \in A_V$; namely $e_V$ is the identity matrix $I_{n_V}$. Note that left-multiplication by $e_V$ gives the projections $\Q[G] \rightarrow A_V$ and $H_1(S;\Q) \rightarrow H_1(S;\Q)_V$.

We note that if $G$ acts freely, then by Theorem \ref{thm:H1asrep} the isotypic component $H_1(S;\Q)_V$ is a free $A_V$-module.  However, this module need not be free if $G$ does not act freely, as we see in Section 3.2.

\para{The Reidemeister pairing}
Let $\hat{i}:H_1(S;\Q) \times H_1(S;\Q) \rightarrow \Q$ denote the intersection form; this is a non-degenerate alternating bilinear form.  As is well known (see e.g.\ \cite{Loo1} or \cite{GLLM}), one can upgrade $\hat{i}$ to a non-degenerate skew-Hermitian form via the action of $G$.  Namely, we define the \emph{Reidemeister pairing} $\eta: H_1(S;\Q) \times H_1(S;\Q) \rightarrow \Q[G]$ by
\begin{equation*}
\eta(x,y) = \sum_{g \in G} \hat{i}(x, gy)g.
\end{equation*}
The non-degeneracy of $\eta$ follows directly from that of $\hat{i}$.  The group ring $\Q[G]$ has the \emph{inversion involution} $\tau$, defined on elements of $G$ by $\tau(u) = u^{-1}$ and extended linearly.  With respect to $\tau$, the Reidemeister pairing is a skew-Hermitian form, meaning that
\begin{itemize}
\item $\eta$ is $\Q$-bilinear,
\item $\eta(vx, wy) = v\eta(x,y)\tau(w)$ for all $x,y \in H_1(S;\Q)$ and $v,w \in \Q[G]$,
\item $\eta(y,x) = -\overline{\eta(x,y)}$ for all $x,y \in H_1(S;\Q)$.
\end{itemize}

Note that $\tau$ preserves each simple factor $A_V \subseteq \Q[G]$ \cite[Lemma~3.2]{GLLM}.  Moreover, the Reidemeister pairing restricts to a non-degenerate $A_V$-valued skew-Hermitian form on $H_1(S;\Q)_V$.  It is $A_V$-valued because for every $x,y \in H_1(S;\Q)_V$,
\begin{equation*}
\eta(x,y) = \eta(e_Vx, e_Vy) = e_V\eta(x,y)\tau(e_V) \in A_V,
\end{equation*}
and it is nondegenerate because an analogous argument shows that the isotypic components are orthogonal with respect to $\eta$.

\subsection{The action of mapping classes}

As mentioned above, the symplectic representation $\Mod(S) \rightarrow \Sp(H_1(S;\Q))$ restricts to a representation $\Phi:\Mod(S)^G \rightarrow \Sp(H_1(S;\Q))^G$, where the superscript denotes the centralizer.  Note that an automorphism of $H_1(S;\Q)$ that commutes with $G$ will preserve $\hat{i}$ if and only if it preserves the Reidemeister pairing, i.e.\ 
\begin{equation*}
\Sp(H_1(S;\Q))^G = \Aut_{\Q[G]}(H_1(S;\Q), \eta).
\end{equation*}
We let $\G = \Aut_{\Q[G]}(H_1(S;\Q),\eta)$.  Now $\G$ is a linear algebraic group defined over $\Q$.  Since $\Mod(S)$ preserves the integer lattice $H_1(S;\Z) \subseteq H_1(S;\Q)$,
\begin{equation*}
\Im(\Phi) \subseteq \Aut_{\Z[G]}(H_1(S;\Z),\eta) =: \G(\Z).
\end{equation*}
The group $\G(\Z)$ is precisely the integer points of the algebraic group $\G$.

Since every $\Q[G]$-automorphism preserves the isotypic components,
\begin{equation*}
\G = \prod_{V \in \Irr(G)} \Aut_{A_V}(H_1(S;\Q)_V, \eta).
\end{equation*}
We let $\G_V = \Aut_{A_V}(H_1(S;\Q)_V,\eta)$.  This is an algebraic group defined over $K$, where $K$ is the subfield of the center of $D_V$ that is fixed by $\tau$.  In most cases we study, $K=\Q$.  Post-composing $\Phi$ with the projections $\pi_V:\G \rightarrow \G_V$, we obtain maps $\Phi_V:\Mod(S)^G \rightarrow \G_V$.  Now, let $\Lambda_V = A_V \cap \Z[G]$ and let
\begin{equation*}
H_1(S;\Z)_V = H_1(S;\Q)_V \cap H_1(S;\Z).
\end{equation*}
The image of $\G(\Z)$ under the projection $\pi_V$ is given by
\begin{equation*}
\pi_V(\G(\Z)) =  \Aut_{\Lambda_V}(H_1(S;\Z)_V, \eta) =: \G_V(\mathcal{O}_K),
\end{equation*}
where $\mathcal{O}_K$ is the ring of integers of $K$. In particular, $\G_V(\mathcal{O}_K)$ is the integer points of the $K$-algebraic group $\G_V$, and since $\Mod(S)$ preserves $H_1(S;\Z)_V$, $\Im(\Phi_V) \subseteq \G_V(\mathcal{O}_K)$.

A subgroup of an algebraic group is \emph{arithmetic} if it has finite index in the integer points of its Zariski closure (in particular, finite subgroups are always arithmetic).  
It need not be that $\Im(\Phi_V)$ is Zariski dense in $\G_V$ (e.g.\ if $\Mod(S)^G$ is finite and $\G_V(\mathcal{O}_K)$ is not), but it will be in every interesting case of Theorem \ref{thm:main}.  Namely, we will show the following, from which Theorem \ref{thm:main} follows.

\begin{theorem}
For every action of a finite group $G$ on a surface $S$ of genus $g \leq 3$ and every $V \in \Irr(G)$, the image of the representation $\Phi_V:\Mod(S)^G \rightarrow \G_V(\mathcal{O}_K)$ is either finite or finite index.
\end{theorem}

\para{Representation of the quotient}
Theorem \ref{thm:main} yields arithmetic representations of $\Mod(S^\circ/G)$ as follows.

Fix a base point $*$ on $S^\circ/G$ and a base point $\widetilde{*}$ on $S$ in the fiber over $*$.  Let $\Mod(S^\circ/G, *)$ denote the based mapping class group of $S^\circ/G$, i.e.\ the group of homeomorphisms that perserve the set of branch points and fix $*$, up to isotopy rel the branch points and $*$.  Note that $\Mod(S^\circ/G, *)$ acts on $\pi_1(S^\circ/G)$ and hence on the finite set $\Hom(\pi_1(S^\circ/G), G)$.  We define the finite-index subgroup
\begin{equation*}
\LMod(S^\circ/G, *) := \Stab(\varphi),
\end{equation*}
where $\varphi:\pi_1(S^\circ/G) \rightarrow G$ is the monodromy homomorphism.

For every $[f] \in \LMod(S^\circ/G, *)$, the homeomorphism $f$ has a unique lift $\widetilde{f}:S^\circ \rightarrow S^\circ$ which fixes $\widetilde{*}$ and commutes with $G$.  Thus, we get a well-defined lifting map $\LMod(S^\circ/G,*) \rightarrow \Mod(S^\circ, \widetilde{*})^G$ and hence a representation
\begin{equation*}
\Psi: \LMod(S^\circ/G, *) \rightarrow \Sp(H_1(S;\Q))^G.
\end{equation*}
The image of $\LMod(S^\circ/G,*) \rightarrow \Mod(S^\circ,\widetilde{*})^G$ is finite-index, and the map $\Mod(S^\circ,\widetilde{*})^G \rightarrow \Sp(H_1(S;\Q))^G$ factors through $\Mod(S)^G$, so $\Im(\Psi)$ is commensurable to $\Im(\Phi)$.  

By \cite[Proposition~8.1]{GLLM}, the map $\Psi$ virtually factors through $\Mod(S^\circ/G)$; the idea is as follows.    Point pushing elements of $\LMod(S^\circ/G,*)$ act on $\Ker(\varphi) = \pi_1(S^\circ)$ by inner automorphisms of $\pi_1(S^\circ/G)$.  Since $\Ker(\varphi)$ is finite-index and $H_1(S^\circ)$ is the abelianization of $\Ker(\varphi)$, point-pushing elements of $\LMod(S^\circ/G,*)$ map to finite order elements under $\Psi$.  Thus, if we take $\Gamma \subseteq \LMod(S^\circ/G,*)$ to be the preimage of a torsion-free finite index subgroup of $\Sp(H_1(S))^G$, then on $\Gamma$ the map $\Psi$ factors though a finite index subgroup of $\Mod(S^\circ/G)$.  One can then take the induced representation to get an arithmetic representation of the full group $\Mod(S^\circ/G)$.

%-------------------------------------------------------------------------------------------------------------------------------------------------------------------

\section{Dividing up the cases}

We now begin the proof of Theorem \ref{thm:main}.  Fix an action of a finite group $G$ on a surface $S$ of genus $g \leq 3$; our goal is to show that the representations $\Phi_V:\Mod(S)^G \rightarrow \G_V$ defined in Section 2 have an arithmetic image.  Theorem \ref{thm:main} only requires work for genus $g=2$ and $g=3$, since $\Mod(S^2)$ is trivial and the symplectic representation $\Mod(T^2) \rightarrow \Sp(H_1(T^2;\Z))$ is an isomorphism.  Moreover, the arithmeticity of an action of $G$ on $S$ depends only on the equivalence class of the action, where two actions $\rho_1,\rho_2:G \rightarrow \Homeo(S)$ are equivalent if there exist $\psi \in \Aut(G)$ and $f \in \Homeo(S)$ such that
\begin{equation*}
\begin{tikzcd}
G \arrow[r, "\rho_1"] \arrow[d, "\psi"] & \Homeo(S) \arrow[d, "g \mapsto fgf^{-1}"] \\
G \arrow[r, "\rho_2"] \arrow[r] & \Homeo(S)
\end{tikzcd}
\end{equation*}
commutes. Thus, it is enough to verify Theorem \ref{thm:main} for a list of representatives of each equivalence class of actions on genus 2 and 3 surfaces; in \cite{Bro}, Broughton provides precisely such a list.  For each case, we show that it is arithmetic in one of five ways:
\begin{enumerate}
\item \emph{Small quotient:} If $S/G$ has genus 0 and 3 or fewer branch points, then $\Mod(S)^G$ is finite and hence the action is automatically arithmetic.
\item \emph{No faithful irreps}: If $H_1(S;\Q)$ does not contain any faithful irreducible subrepresentations of $G$, then the arithmeticity reduces to lower genus cases.
\item \emph{Prior work}: These cases follow directly from another work, namely \cite{ACA}, \cite{DM}, \cite{Mos}, \cite{Loo1}, or \cite{McM}.
\item \emph{Lifting twists}: If the group $\G_V$ is isomorphic to $\SL_2(\Q)$, then we prove arithmeticity by lifting twists (see Section 4).
\item \emph{Direct proof}: If a case does not fall into the above categories, then we provide an individual proof of arithmeticity.
\end{enumerate}

We summarize the cases in Table \ref{table:main}.  About half of the cases on Broughton's list are small quotient cases, which we omit from the table.  The first column is the label from Broughton's list; labels starting with $2$ are genus 2 actions, and labels starting with 3 are genus 3 actions.  For each case, the \emph{branching data} is a tuple $(h; m_1, \ldots, m_n)$ where $h$ is the genus of $S/G$ and the $m_i$ are the orders of the branch points $x_i$; the genus $h$ is omitted when $h=0$, and exponents are used to denote repeated branching orders (e.g.\ $(2^3, 4^2)$ denotes $(0;2,2,2,4,4)$).  The monodromy homomorphism is described by a tuple of elements of $G$, as explained at the start of Section 2 (Broughton calls this a ``generating vector").  Our notation is largely standard, but in particular:
\begin{itemize}
\item $C_m$ denotes the cyclic group of order $m$, and we fix a generator $c$.  For a product of cyclic groups we let $c'$ and $c''$ denote a generator of the second and third factor respectively.
\item $D_m$ denotes the dihedral group of order $2m$, with the presentation
\begin{equation*}
D_m = \la r,s \mid r^m = s^2 = 1, srs^{-1} = r^{-1}\ra.
\end{equation*}
\item $C_4 \circ D_4$ denotes the \emph{central product} of $C_4$ and $D_4$, with the presentation
\begin{equation*}
C_4 \circ D_4 = \la c, r, s \mid c^4 = r^4 = s^2 = [c,r] = [c,s] = 1, c^2 = r^2, srs^{-1} = r^{-1} \ra.
\end{equation*}
\end{itemize}
We note that Broughton's list contains two typos: the monodromy in 3.r.1 should be $(x,y,z,yz,x)$, and the second action in 3.ad.2 does not have a valid monodromy (the correct entry $(x,xzy,z^2,yz)$ is not a surjective homomorphism and corresponds to the case 2.n).

\begin{table}[h!]
\center
\begin{tabular}{|c|c|c|c|c|}
\hline
Case & $G$ & Branching data & Monodromy &  Case type \\
\hline
\rowcolor{lightgray}
2.a & $C_2$ & $(2^6)$  & $(c,c,c,c,c,c)$ & Prior work \\
\hline
2.b & $C_2$ & $(1;2^2)$  & $(1,1,c,c)$ & Direct proof \\
\hline
\rowcolor{lightgray}
2.c & $C_3$ & $(3^4)$  & $(c,c,c^{-1},c^{-1})$ & Prior work \\
\hline
\rowcolor{lightgray}
2.e & $C_4$ & $(2^2,4^2)$  & $(c^2,c^2,c,c^{-1})$ & Prior work \\
\hline
\rowcolor{lightgray}
2.f & $C_2 \times C_2$ & $(2^5)$  & $(c,c,c,c',cc')$ & No faithful irreps \\
\hline
\rowcolor{lightgray}
2.k.1 & $C_6$ & $(2^2,3^2)$  & $(c^3,c^3,c^2,c^4)$ & Prior work \\
\hline
2.k.2 & $D_3$ & $(2^2,3^2)$  & $(s,s,r,r^{-1})$ & Lifting twists \\
\hline
2.n & $D_4$ & $(2^3,4)$  & $(s,sr,r^2,r)$ & Lifting twists \\
\hline
2.s & $D_6$ & $(2^3,3)$  & $(s,sr,r^3,r^2)$ & Lifting twists \\
\hline
\rowcolor{lightgray}
3.a & $C_2$ & $(2^8)$  & $(c,c,c,c,c,c,c,c)$ & Prior work \\
\hline
3.b & $C_2$ & $(1;2^4)$ & $(1,1,c,c,c,c)$ & Direct proof \\
\hline
\rowcolor{lightgray}
3.c & $C_2$ & $(2;-)$ & $(c,1,1,1)$ & Prior work \\
\hline
\rowcolor{lightgray}
3.d & $C_3$ & $(3^5)$ & $(c,c,c,c,c^{-1})$ & Prior work \\
\hline
3.e & $C_3$ & $(1;3^2)$ & $(1,1,c,c^{-1})$ & Direct proof \\
\hline
\rowcolor{lightgray}
3.f.I & $C_4$ & $(4^4)$ & $(c,c,c,c)$ & Prior work \\
\hline
\rowcolor{lightgray}
3.f.II & $C_4$ & $(4^4)$ & $(c,c,c^{-1},c^{-1})$ & Prior work \\
\hline
\rowcolor{lightgray}
3.g & $C_4$ & $(2^3,4^2)$ & $(c^2,c^2,c^2,c,c)$ & Prior work \\
\hline
\rowcolor{lightgray}
3.h & $C_2 \times C_2$ & $(2^6)$ & multiple different actions & No faithful irreps \\
\hline
3.i.1 & $C_4$ & $(1;2^2)$ & $(c,1,c^2,c^2)$ & Lifting twists \\
\hline
\rowcolor{lightgray}
3.i.2 & $C_2 \times C_2$ & $(1;2^2)$ & $(c',1,c,c)$ & No faithful irreps \\
\hline
\rowcolor{lightgray}
3.j & $C_6$ & $(2^2,6^2)$ & $(c^3,c^3,c,c^{-1})$ & Prior work \\
\hline
\rowcolor{lightgray}
3.k & $C_6$ & $(2,3^2,6)$ & $(c^3,c^4,c^4,c)$ & Prior work \\
\hline
3.m & $D_3$ & $(2^4,3)$ & $(s,s,s,sr^{-1},r)$ & Lifting twists \\
\hline
3.n & $D_3$ & $(1; 3)$ & $(s,sr,r)$ & Lifting twists \\
\hline
\rowcolor{lightgray}
3.q.1 & $C_2 \times C_4$ & $(2^2,4^2)$ & multiple different actions & No faithful irreps \\
\hline
3.q.2 & $D_4$ & $(2^2,4^2)$ & $(s,s,r^{-1},r)$ & Lifting twists \\
\hline
\rowcolor{lightgray}
3.r.1 & $(C_2)^3$ & $(2^5)$ & $(c,c',c'',c'c'',c)$ & No faithful irreps \\
\hline
3.r.2 & $D_4$ & $(2^5)$ & $(s,s,sr, sr^3, r^2)$ & Lifting twists \\
\hline
3.s.1 & $Q_8$ & $(1;2)$ & $(i,j,-1)$ & Direct proof \\
\hline
3.s.2 & $D_4$ & $(1;2)$ & $(s,sr,r^2)$ & Lifting twists \\
\hline
3.y & $D_6$ & $(2^3,6)$ & $(s,sr^2,r^3,r)$ & Lifting twists \\
\hline
3.z & $A_4$ & $(2^2,3^2)$ & $((12)(34),(12)(34),(123),(321))$ & Lifting twists \\
\hline
\rowcolor{lightgray}
3.ad.1 & $C_2 \times D_4$ & $(2^3,4)$ & $(c,s,scr^{-1},r)$ & No faithful irreps \\
\hline
3.ad.2 & $C_4 \circ D_4$ & $(2^3,4)$ & $(s,c r,sr,c^{-1})$ & Direct proof \\
\hline
3.al & $S_4$ & $(2^3,3)$ & $((12),(23),(13)(24),(243))$ & Lifting twists \\
\hline
\end{tabular}
\caption{Summary of cases.  The 18 easier cases are grayed out.}
\label{table:main}
\end{table}

In Section 3.1, we observe some simple reductions that handle the ``small quotient" (omitted from table) and ``no faithful irreps" cases.  In Section 3.2, we determine which cases follow from prior work.  In Section 3.3, we determine precisely the target group $\G_V$ in all remaining cases.  In Section 3.4, we handle the ``direct proof" cases.  The remaining cases are then handled by the process of lifting twists, which we describe in Section 4.

\subsection{First reductions} \phantom{}

\para{Small quotient}  First, we can verify that ``small quotient" cases are automatically arithmetic.

\begin{lemma}
If $S/G$ has genus $h=0$ and $n \leq 3$ branch points, then $\Mod(S)^G$ is finite.
\end{lemma}
\begin{proof}
Let $\textrm{SMod}(S) \subseteq \Mod(S)$ be the subgroup of \emph{symmetric mapping classes}, i.e.\ mapping classes represented by lifts of homeomorphisms of $S^\circ/G$.    Let $\Gamma \subseteq \Mod(S^\circ/G)$ be the subgroup of mapping classes that lift along $S^\circ \rightarrow S^\circ/G$.  Lifting mapping classes gives a well-defined surjection $\Gamma \twoheadrightarrow \textrm{SMod}(S)/G$.  Since $\Mod(S^\circ/G)$ is finite, $G$ is finite, and $\Mod(S)^G \subseteq \textrm{SMod(S)}$, the lemma follows. 
\end{proof}

\para{No faithful irreps} Next, we show that if $V \in \Irr(G)$ is not a faithful representation, then the representation $\Phi_V$ factors through a lower genus case.  In particular, if $G$ has no faithful irreducible representations, then the arithmeticity of a $G$-action reduces entirely to lower genus cases.

\begin{lemma}\label{lem:reduction_by_transfer}
Let $V \in \Irr(G)$, and let $N \subseteq G$ be its kernel (so $V$ is also a representation of $G/N$).  Then there is an isomorphism $\Aut_{A_V}(H_1(S;\Q)_V) \cong \Aut_{A_V}(H_1(S/N;\Q)_V)$ such that the diagram
\begin{equation*}
\begin{tikzcd}
\Mod(S)^G \arrow[r] \arrow[d] & \Aut_{A_V}(H_1(S;\Q)_V) \arrow[d, "\cong"] \\
\Mod(S/N)^{G/N} \arrow[r] & \Aut_{A_V}(H_1(S/N;\Q)_V)
\end{tikzcd}
\end{equation*}
commutes.
\end{lemma}
\begin{proof}
Note that the map $q:S^\circ \rightarrow S^\circ/N$ is a $G$-equivariant covering map (where $G$ acts on $S/N$ via $G/N$).  Recall the \emph{transfer homomorphism} $q^*:H_1(S^\circ/N;\Q) \rightarrow H_1(S^\circ;\Q)^N$ defined by sending a chain to the sum of its lifts; the map $q^*$ is in fact an isomorphism with inverse $(1/\v N \v)q_*$.  Let $B = \Ker(H_1(S^\circ;\Q) \rightarrow H_1(S;\Q))$ and let $\overline{B} = \Ker(H_1(S^\circ/N;\Q) \rightarrow H_1(S/N;\Q))$.  Since we are working over $\Q$, the invariants functor is exact, so $\Ker(H_1(S^\circ;\Q)^N \rightarrow H_1(S;\Q)^N) = B^N$.  Since $q^*$ takes $\overline{B}$ to $B^N$, it descends to an isomorphism $ H_1(S/N;\Q) \cong H_1(S;\Q)^N$.  The lemma follows from the observation that $H_1(S;\Q)_V \subseteq H_1(S;\Q)^N$.
\end{proof}

Note that the cases in Table \ref{table:main} where $G$ has no faithful irreducible representations are precisely those where $G$ is a product.

\subsection{Prior work} \phantom{}

\para{Hyperelliptic involution}
Note that cases 2.a and 3.a are precisely the action of the hyperelliptic involution.  In the genus 2 case, the hyperelliptic involution is central, and so this case reduces to the symplectic representation $\Mod(S) \rightarrow \Sp(H_1(S;\Q))$.  In the genus 3 case, the action is known to be arithmetic by A'Campo \cite{ACA}.

Note also in the genus 2 cases, if $G$ is of the form $H \times C_2$ where $C_2$ acts by the hyperelliptic involution, then the arithmeticity of the $G$-action reduces to the arithmeticity of the $H$-action.  This applies in particular to cases 2.f, 2.k.1, and 2.s.

\para{Cyclic covers of the sphere}
Suppose $G = C_m = \la c \ra$ and fix a cover $S \rightarrow S/G$ such that $S/G$ has genus 0 and $n$ branch points.  Let $c^{k_i} \in C_m$ be the image of the $i$th branch point under the monodromy homomorphism $\varphi:\pi_1(S/G) \rightarrow C_m$, where $k_i \in \{0, \ldots, m\}$.  In this case, the representation $\Phi$ arises as the monodromy of the family of curves
\begin{equation*}
y^m = (x-b_1)^{k_1} \cdots (x-b_{n-1})^{k_{n-1}},
\end{equation*}
where $b_1, \ldots, b_{n-1}$ are distinct points of $\C$.  

To see this, let $\PConf_{n-1}$ be the set of $(b_1, \ldots, b_{n-1}) \in \C^{n-1}$ such that $b_i \neq b_j$ for all $i$ and $j$, and for each $B = (b_i) \in \PConf_{n-1}$, let $p_B(x) = (x-b_1)^{k_1} \cdots (x-b_{n-1})^{k_{n-1}}$ and let $Y_B = \C - \{b_1, \ldots, b_{n-1}\}$.  Then we have a family of curves
\begin{equation*}
Y_B \rightarrow \{(w,x,B) \in \C^* \times \C \times \PConf_{n-1} \mid w = p_B(x)\} \rightarrow \PConf_{n-1}.
\end{equation*}
Recall that $\pi_1(\PConf_{n-1})$ is the pure braid group on $n-1$ strands $\PB_{n-1}$.  Viewing $\PB_{n-1}$ as the mapping class group of the $(n-1)$-times punctured disk $D_{n-1}$, the monodromy action $\PB_{n-1} \rightarrow \PMod(S_B)$ is the usual ``capping homomorphism" obtained by the inclusion $D_{n-1} \hookrightarrow S_{0,n}$.  The family $Y_B$ has a fiberwise branched cover
\begin{equation*}
S_B \rightarrow \{(y,x,B) \in \C^* \times \C \times \PConf_{n-1} \mid y^m = p_B(x)\} \rightarrow \PConf_{n-1}.
\end{equation*}
The cover $S_B \rightarrow Y_B$ is precisely our cover $S \rightarrow S/G$ (the $n$th branch point $x_n$ is $\infty$).  Since the monodromy of family $S_B$ is obtained by lifting the monodromy of family $Y_B$, we conclude that the monodromy of $S_B$ is commensurable to the map $\Psi$ defined in Section 2.2, and hence commensurable to $\Phi$.

This is precisely the family studied by Deligne-Mostow in \cite{DM} and \cite{Mos}.  However, they study in particular the monodromy action on $H_1(S;\C)$ (technically they study the action on $H^1(S;\C)$, but we can simply dualize).  So, we clarify the relationship between $\Q$ and $\C$ coefficients.  The intersection pairing $\hat{i}$ on $H_1(S;\Q)$ extends to a skew-Hermitian form on $H_1(S;\C)$, namely the dual of the form defined on $H^1(S;\C)$ by
\begin{equation*}
(\alpha, \beta) \mapsto \int_S \alpha \wedge \overline{\beta}.
\end{equation*}
Thus, the Reidemeister pairing $\eta$ also extends naturally to a $\C[G]$-valued form on $H_1(S;\C)$.  The embedding $H_1(S;\Q) \hookrightarrow H_1(S;\C)$ yields a complexification map
\begin{equation*}
\Theta:\Sp(H_1(S;\Q))^G \hookrightarrow \U(H_1(S;\C))^G .
\end{equation*}
In particular, the map $\Phi:\Mod(S)^G \rightarrow \Sp(H_1(S;\Q))^G$ can be viewed as a map $\Mod(S)^G \rightarrow \U(H_1(S;\C))^G$.  The generator $c$ of $G = C_m$ acts by a finite-order automorphism of $H_1(S;\C)$ which is automatically diagonalizable.  Each eigenvalue is an $m$th root of unity.  We let $\zeta = e^{2\pi i/m}$, and denote the $\zeta^k$-eigenspace by $H_1(S;\C)_{\zeta^k}$.  Since $\U(H_1(S;\C))^G$ preserves the eigenspaces,
\begin{equation*}
\U(H_1(S;\C))^G \cong \prod_{k=1}^m \U(H_1(S;\C)_{\zeta^k}).
\end{equation*}
First, we check that this decomposition matches our isotypic component decomposition of $\G = \Sp(H_1(S;\Q))^G$.  Recall that the unique faithful irreducible $\Q$-representation of $C_m$ is the $\Q$-vector space $\Q(\zeta)$, where $c$ acts by $\zeta$.

\begin{lemma}\label{lem:complexification}
The complexification map $H_1(S;\Q) \rightarrow H_1(S;\C)$ takes the isotypic component $H_1(S;\Q)_{\Q(\zeta)}$ to the sum of the eigenspaces
\begin{equation*}
\bigoplus_{k \in (\Z/m\Z)^\times} H_1(S;\C)_{\zeta^k}.
\end{equation*}
Moreover, on each eigenspace $H_1(S;\C)_{\zeta^k}$, the Reidemeister pairing restricts to a multiple of $\hat{i}$ (via the natural projection $\C[G] \rightarrow \C$ given by mapping the generator $c$ to $\zeta^k$).
\end{lemma}
\begin{proof}
The first statement is purely a fact from the representation theory of $G = C_m$.  On the one hand, the Wedderburn decomposition of $\Q[G]$ is realized by the well-known isomorphisms
\begin{equation*}
\Q[G] \cong \Q[x]/(x^m-1) \cong \prod_{d \mid m} \Q[x]/(\Phi_d(x))
\end{equation*}
where $\Phi_d(x)$ is the $d$th cyclotomic polynomial.
On the other hand,
\begin{equation*}
\C[G] \cong \prod_{k=1}^m \C
\end{equation*}
where the projection onto the $k$th factor sends $c$ to $\zeta^k$.  Thus the inclusion $\Q[G] \hookrightarrow \C[G]$ maps
\begin{equation*}
\Q[x]/(\Phi_m(x)) \rightarrow \prod_{k \in (\Z/m\Z)^\times} \C
\end{equation*}
via the maps $x \mapsto \zeta^k$, which precisely correspond to the embeddings $\Q(\zeta) \hookrightarrow \C$.  In particular, if we let $e_{\Q(\zeta)}$ be the idempotent of $\Q(\zeta) \subseteq \Q[G]$ and let $e_k$ be the idempotents of $\C[G]$, then
\begin{equation*}
e_{\Q(\zeta)} = \sum_{k \in (\Z/d\Z)^\times} e_k.
\end{equation*}
Since the projections $H_1(S;\Q) \rightarrow H_1(S;\Q)_{\Q(\zeta)}$ and $H_1(S;\C) \rightarrow \bigoplus H_1(S;\C)_{\zeta^k}$ are given by multiplication by idempotents, the first statement of the lemma follows.

To prove the second statement, take any $x,y \in H_1(S;\C)_{\zeta^k}$.  Viewing $\eta(x,y)$ as an element of $\C[G]$, 
\begin{equation*}
\eta(x,y) = e_k\eta(x,y) = e_k \sum_{\ell=1}^m \hat{i}(x, c^\ell y)c^\ell
= e_k \sum_{\ell=1}^m \hat{i}(x, \zeta^{k\ell}y)c^\ell.
\end{equation*}
Under the isomorphism $e_k\C[G] \cong \C$, this maps to
\begin{equation*}
\sum_{\ell=1}^m \hat{i}(x, \zeta^{k\ell}y)\zeta^{k\ell}
= \sum_{\ell=1}^m \hat{i}(x, y)\zeta^{-k\ell}\zeta^{k\ell}
= m\hat{i}(x,y).
\end{equation*}
\end{proof}

Lemma \ref{lem:complexification} tells us that $\Theta:\G \hookrightarrow U(H_1(S;\C))^G$ takes the subgroup 
\begin{equation*}
\G_{\Q(\zeta)} = \Aut_{\Q[G]}(H_1(S;\Q)_{\Q(\zeta)}, \eta)
\end{equation*}
to the subgroup
\begin{equation*}
\prod_{k \in (\Z/m\Z)^\times} \U(H_1(S;\C)_{\zeta^k}).
\end{equation*}
Observe that the eigenspace $H_1(S;\C)_{\zeta^{-k}}$ is conjugate to the eigenspace $H_1(S;\C)_{\zeta^k}$. Recall that if $f$ is an automorphism of a $\Q$-vector space $W$, then the induced map $f_\C$ on the complexification $W_\C$ commutes with conjugation.  This means that the action of $\Theta(\G_{\Q(\zeta)})$ on $H_1(S;\C)_{\zeta^{-k}}$ is determined by its action on $H_1(S;\C)_{\zeta^k}$, and thus $\Theta(\G_{\Q(\zeta)})$ embeds into the product
\begin{equation*}
\mathcal{H} = \prod_{k \in (\Z/m\Z)^\times/\pm 1} \U(H_1(S;\C)_{\zeta^k}).
\end{equation*}
The isotypic component $H_1(S;\Q)_{\Q(\zeta)}$ is isomorphic to $\Q(\zeta)^\ell$ for some $\ell$, and the group $\G_{\Q(\zeta)}$ is a $K$-defined algebraic group isomorphic to $\U(h;\Q(\zeta))$, where $K = \Q(\zeta + \zeta^{-1})$ and $h$ is a skew-Hermitian form on $\Q(\zeta)^\ell$ (given by the Reidemeister pairing).  The embedding of $\Theta(\G_{\Q(\zeta)})$ corresponds to the Weil restriction of scalars; in particular, if we let $S^\infty$ denote the set of embeddings $\Q(\zeta) \hookrightarrow \C$ up to conjugation, then the map $\Theta(\G_{\Q(\zeta)}) \hookrightarrow \mathcal{H}$ corresponds to the map
\begin{equation*}
\U(h; \Q(\zeta)) \hookrightarrow \prod_{\sigma \in S^\infty} \U(\sigma(h))
\end{equation*}
given by sending a matrix to the tuple of its Galois conjugates.

Now, we can recall Deligne and Mostow's results.  For $1 \leq i \leq n$ let $\mu_i = k_i/m$.  Suppose that the following two conditions hold:
\begin{enumerate}
\item $\sum_{i=1}^n k_i = 2$,
\item For each $i,j$ such that $\mu_i+\mu_j < 1$, $(1 - \mu_i - \mu_j)^{-1} \in \Z$ if $i \neq j$ and $(1 - \mu_i - \mu_j)^{-1} \in \frac{1}{2}\Z$ if $i = j$.
\end{enumerate}
Under these hypotheses, Deligne and Mostow prove that $\U(H_1(S;\C)_\zeta) \cong U(1,n-3)$ and the image of $\Theta(\G_{\Q(\zeta)})$ in $\U(H_1(S;\C)_\zeta)$ is a lattice.

We note that if $m$ is equal to 3, 4, or 6, then $\Q(\zeta + \zeta^{-1}) = \Q$ and there is only one embedding $\Q(\zeta) \hookrightarrow \C$ up to conjugation.  Thus, the product $\mathcal{H}$ contains only the single factor $\U(H_1(S;\C)_\zeta)$ and the embedding $\G_{\Q(\zeta)} \rightarrow \Theta(\G_{\Q(\zeta)}) \rightarrow \mathcal{H}$ is an isomorphism.  In these cases, if the hypothesis of Deligne and Mostow holds, we conclude that $\Im(\Phi_{\Q(\zeta)})$ is arithmetic.  This applies to every cyclic action on Table \ref{table:main} with a genus 0 quotient except cases 2.a, 3.a, and 3.f.I.

\para{Remaining cases}
There are two remaining cases that follow from prior work:
\begin{enumerate}
\item Case 3.f.I, which follows directly from McMullen \cite[Corollary~11.7]{McM}.
\item Case 3.c, which follows directly from Looijenga \cite{Loo1}.
\end{enumerate}

\subsection{Determining the target algebraic group}
Using the character formula in Theorem \ref{thm:H1asrep}, one can check that in every case in Table \ref{table:main}, $H_1(S;\Q)$ contains at most one faithful irreducible representation.  By Lemma \ref{lem:reduction_by_transfer}, this is the only isotypic component we need to consider in each action.  So, we summarize the following information for each remaining case of Table \ref{table:main}:
\begin{enumerate}
\item the faithful representation $V \in \Irr(G)$
\item the corresponding simple factor $M_n(D)$ of $\Q[G]$, the restriction of the inversion involution $\tau$ to $M_n(D)$, and the corresponding fixed subfield $K$.
\item the isotypic component $H_1(S;\Q)_V$
\item the Reidemeister pairing $\eta$ on $H_1(S;\Q)_V$
\item the automorphism group $\G_V = \Aut_A(H_1(S;\Q)_V, \eta)$ and its integer points $\G(\mathcal{O}_K)$.
\end{enumerate}

\para{Cyclic covers}
Suppose $G = C_m = \la c \ra$.
\begin{enumerate}
\item The unique faithful $V \in \Irr(G)$ is the $\Q$-vector space $\Q(\zeta)$ where $\zeta = e^{2\pi i/m}$ and $c$ acts by $\zeta$.  Note that for $m = 2$ this is also called the \emph{sign representation}, denoted $\Q_-$.
\item The corresponding simple factor is the field $\Q(\zeta)$ ($n=1$, $D=\Q(\zeta_m)$); the projection $\Q[G] \rightarrow \Q(\zeta)$ is given by $c \mapsto \zeta$.  The involution $\tau$ restricts to complex conjugation on $\Q(\zeta)$, and $K = \Q(\zeta + \zeta^{-1})$.
\item The isotypic component is $\Q(\zeta)^\ell$ for some $\ell \in \Z_{>0}$.
\item The Reidemeister pairing is a skew-Hermitian form $h$ on the $\Q(\zeta)$-vector space $\Q(\zeta)^\ell$.  If $m=2$, $h$ is in fact a symplectic form on $\Q^\ell$.
\item $\G_V \cong \U(h;\Q(\zeta))$ and $\G_V(\mathcal{O}_K) \cong \U(h;\Z[\zeta])$.  If $m=2$, then in particular $K=\Q$, $\G_V \cong \Sp(h; \Q)$, and $\G_V(\Z) \cong \Sp(h; \Z)$.
\end{enumerate}

\para{Dihedral covers}
Suppose $G = D_m$ for $m \in \{3,4,6\}$.
\begin{enumerate}
\item The unique faithful $V \in \Irr(G)$ is the two dimensional representation
\begin{equation*}
r \mapsto
\begin{pmatrix}
0 & -1 \\ 1 & 2\cos(2\pi/m)
\end{pmatrix},\ 
s \mapsto
\begin{pmatrix}
0 & 1 \\ 1 & 0
\end{pmatrix}.
\end{equation*}

\item
The corresponding simple factor is $M_2(\Q)$; the projection $\Q[G] \rightarrow M_2(\Q)$ is precisely given by the representation $V$.  For $m = 4$, the involution $\tau$ restricts to the transpose involution $B \mapsto B^t$.  For $m=3$ and $m=6$, $\tau$ does not directly restrict to the transpose, but we claim it is conjugate over $\Q$ to the transpose.  From this claim, we conclude that $K = \Q$.  The claim follows because any two involutions on a simple $\Q$-algebra that agree on the center are conjugate \cite[Lemma~2.10]{PR}, so it is enough to check that $\tau$ fixes the center of $M_2(\Q)$.  But this true since $\tau$ is conjugate over $\R$ to the transpose, as $V$ is conjugate over $\R$ to an orthogonal representation.

\item
In every case we consider, the isotypic component is $V^2$, which is isomorphic as an $M_2(\Q)$-module to $M_2(\Q)$ itself.

\item
The Reidemeister pairing is determined by the single matrix $F:= \eta(I_2, I_2) \in M_2(\Q)$.  The skew-Hermitian condition implies that $F = -\tau(F)$; in the case $m=4$ where $\tau$ is the transpose, this implies that $F$ is a multiple of the matrix $J = \left(\begin{smallmatrix} 0 & 1 \\ -1 & 0 \end{smallmatrix} \right)$.

\item
Note that $\Aut_{M_2(\Q)}(M_2(\Q)) = \GL_2(\Q)$ and
\begin{equation*}
\G_V = \{B \in \GL_2(\Q) \mid BF\tau(B) = F\}.
\end{equation*}
In the case $m=4$, the condition $BF\tau(B) = F$ is equivalent to $BJB^t = J$, and we directly obtain that $\G_V \cong \Sp_2(\Q) = \SL_2(\Q)$ and $\G_V(\Z) \cong \SL_2(\Z)$.  In the cases $m=3$ and $m=6$, the fact that $\tau$ is conjugate over $\Q$ to the transpose implies that that $\G_V$ is conjugate over $\Q$ to $\Sp_2(\Q) = \SL_2(\Q)$, and hence $\G_V \cong \SL_2(\Q)$ and $\G_V(\Z) \cong \SL_2(\Z)$.
\end{enumerate}

\para{Case 3.s.1}
In this case $G = Q_8$ (c.f.\ \cite[Section~9.6]{GLLM}).
\begin{enumerate}
\item The representation $V \in \Irr(G)$ is a 4-dimensional representation $G \rightarrow \SL_2(\Q(i))$ given by
\begin{equation*}
i \mapsto 
\begin{pmatrix}
0 & 1 \\ -1 & 0
\end{pmatrix},\ 
j \mapsto
\begin{pmatrix}
i & 0 \\ 0 & -i
\end{pmatrix}.
\end{equation*}

\item
The corresponding simple factor is the division algebra
\begin{equation*}
D = \left\{
\begin{pmatrix}
\alpha & \beta \\ -\overline{\beta} & \overline{\alpha}
\end{pmatrix}
\mid \alpha,\beta \in \Q(i)\right\}.
\end{equation*}
The projection $\Q[G] \rightarrow D$ is precisely given by the representation $V$.  The inversion involution $\tau$ restricts to the conjugate transpose $B \mapsto B^*$, and $K = \Q(i+i^{-1}) = \Q$.

\item
The isotypic component is $V$, which is isomorphic as a $D$-module to $D$ itself.

\item
The Reidemeister pairing is determined by the single matrix $F := \eta(I_2, I_2) \in D$.  The skew-Hermitian property implies that $F^* = -F$.

\item
Note that $\Aut_D(D) = D^\times$.  Now $\G_V = \{B \in D^\times \mid BFB^* = F\}$, and $\G_V(\Z)$ is the subgroup with entries in $\Z[i]$.
\end{enumerate}

\para{Non-free cases}
The final three cases (3.z, 3.ad.2, and 3.al) are slightly different because the isotypic component $H_1(S;\Q)_V$ is not a free module over the corresponding simple factor.  We will first examine such a scenario in generality, and then we examine our specific cases below.  

Suppose $V \in \Irr(G)$ appears in $H_1(S;\Q)$ with multiplicity $k$.  Suppose further that the simple factor is $M_n(L)$ for a number field $L$ where $k < n$.  Then $H_1(S;\Q)_V$ is isomorphic as an $M_n(L)$-module to the set of $n \times k$ matrices with the left action of $M_n(L)$.  In particular $H_1(S;\Q)_V$ is spanned over $M_n(L)$ by the block matrix
\begin{equation*}
E := 
\begin{pmatrix}
I_k \\ 0
\end{pmatrix}.
\end{equation*}

First, we can identify the endomorphisms of $H_1(S;\Q)_V$.

\begin{lemma}
As rings, $\End_{M_n(L)}(H_1(S;\Q)_V) \cong M_k(L)^{\textrm{op}}$.
\end{lemma}
\begin{proof}
We have a map $M_k(L) \rightarrow \End_{M_n(L)}(H_1(S;\Q))$ given by $C \mapsto T_C$, where $T_C$ is the unique $M_k(L)$-linear map that multiplies $E$ by the following block matrix:
\begin{equation*}
T_C(E) =
\begin{pmatrix}
C & 0 \\ 0 & 0
\end{pmatrix}E.
\end{equation*}
To see that $T_C$ is a well-defined endomorphism, suppose that $AE = BE$ for $A,B \in M_n(L)$.  Then $A$ and $B$ have the same first $k$ columns, and so
\begin{equation*}
T_C(AE) = AT_C(E) = A
\begin{pmatrix}
C & 0 \\ 0 & 0
\end{pmatrix}E
= B
\begin{pmatrix}
C & 0 \\ 0 & 0
\end{pmatrix}
E
= BT_C(E)
= T_C(BE).
\end{equation*}
We can see that this map is injective, additive, and satisfies $T_{C_1C_2} = T_{C_2}T_{C_1}$.  To see that it is surjective, note that every $T \in \End_{M_k(L)}(H_1(S;\Q)_V)$ is determined by where it sends $E$, and so it suffices to show that $T$ must send $E$ to a $n \times k$ matrix where the bottom $n-k$ rows are all zero.  But this follows since
\begin{equation*}
\begin{pmatrix}
I_k & 0 \\ 0 & 0
\end{pmatrix}
T(E)
= T\left(
\begin{pmatrix}
I_k & 0 \\ 0 & 0
\end{pmatrix}E
\right)
= T(E).  \qedhere
\end{equation*}
\end{proof}

Now, we can better understand the Reidemeister pairing $\eta$ in some particular cases.  Note first that the ``Gram matrix" of $\eta$ on $H_1(S;\Q)_V$ is the $1 \times 1$ matrix $\left( \eta(E,E) \right) \in M_1(M_n(L))$.  In other words, $\eta$ is completely determined by the matrix $F := \eta(E,E) \in M_n(L)$.  Note that the relation
\begin{equation*}
F = \eta(E,E) = \eta\left(
\begin{pmatrix}
I_k & 0 \\ 0 & 0
\end{pmatrix}E,E \right)
=
\begin{pmatrix}
I_k & 0 \\ 0 & 0
\end{pmatrix}
\eta(E,E)
=
\begin{pmatrix}
I_k & 0 \\ 0 & 0
\end{pmatrix}F
\end{equation*}
implies that
\begin{equation*}
F = \begin{pmatrix}
\star \\ 0
\end{pmatrix}
\end{equation*}
for some $k \times n$ matrix $\star$.  Now, suppose first that $\tau$ restricts to the transpose.  In this case, the relation $F = -\tau(F)$ implies that $F$ is a multiple of 
\begin{equation*}
\begin{pmatrix}
J & 0 \\ 0 & 0
\end{pmatrix}
\end{equation*}
where $J \in M_k(L)$ is a skew-symmetric matrix.  If instead $\tau$ restricts to the conjugate transpose, then the relation $F = -\tau(F)$ implies that $F$ is a multiple of
\begin{equation*}
\begin{pmatrix}
J & 0 \\ 0 & 0
\end{pmatrix}
\end{equation*}
where $J \in M_k(L)$ is a skew-Hermitian matrix.

\para{Case 3.z}
In this case $G = A_4$.
\begin{enumerate}
\item
The representation $V \in \Irr(G)$ is the 2-dimensional standard representation.  In particular $V$ is the restriction of the representation of $S_4$ given by
\begin{equation*}
(1,2) \mapsto
\begin{pmatrix}
0 & 1 & 0 \\ 1 & 0 & 0 \\ 0 & 0 & 1
\end{pmatrix},\ 
(2,3) \mapsto 
\begin{pmatrix}
0 & 0 & -1 \\ 0 & 1 & 0 \\ -1 & 0 & 0
\end{pmatrix},\ 
(3,4) \mapsto
\begin{pmatrix}
0 & -1 & 0 \\ -1 & 0 & 0 \\ 0 & 0 & 1
\end{pmatrix}.
\end{equation*}

\item
The corresponding simple factor is $M_3(\Q)$; the projection $\Q[G] \rightarrow M_3(\Q)$ is precisely given by the representation $V$.  The inversion involution $\tau$ restricts to the transpose involution $B \mapsto B^t$, and $K = \Q$.

\item
The isotypic component is $V^2$.  As an $M_2(\Q)$-module, we can take $V^2$ to be the space of $3 \times 2$ matrices with the left action of $M_2(\Q)$.  

\item
The Reidemeister pairing is completely determined by the matrix $F := \la E, E \ra \in M_3(\Q)$.  Since $\tau$ restricts to the transpose, we conclude that $F$ is a multiple of
\begin{equation*}
\begin{pmatrix}
0 & 1 & 0 \\ -1 & 0 & 0 \\ 0 & 0 & 0
\end{pmatrix}.
\end{equation*}

\item
From the isomorphism $\End_{M_3(\Q)}(V^2) \cong M_2(\Q)$ and our computation of $F$, we conclude as in the dihedral cases that $\G_V \cong \SL_2(\Q)$ and $\G_V(\Z) \cong \SL_2(\Z)$.
\end{enumerate}

\para{Case 3.ad.2}
In this case $G = C_4 \circ D_4$.
\begin{enumerate}
\item
The representation $V \in \Irr(G)$ is the 4-dimensional representation $G \rightarrow \GL_2(\Q(i))$ given by
\begin{equation*}
c \mapsto 
\begin{pmatrix}
i & 0 \\ 0 & i
\end{pmatrix},\ 
r \mapsto
\begin{pmatrix}
0 & -1 \\ 1 & 0
\end{pmatrix},\ 
s \mapsto
\begin{pmatrix}
0 & 1 \\ 1 & 0
\end{pmatrix}.
\end{equation*}

\item
The corresponding simple factor is $M_2(\Q(i))$.  The projection $\Q[G] \rightarrow M_2(\Q(i))$ is precisely given by $V$.  The involution $\tau$ descends to the conjugate transpose $B \mapsto B^*$.

\item
The isotypic component is $V$.  As an $M_2(\Q(i))$-module, we can take $V$ to be the space of $2 \times 1$ matrices with the left action of $M_2(\Q(i))$.

\item
The Reidemeister pairing is completely determined by the matrix $F := \la E, E \ra \in M_2(\Q(i))$.  Since $\tau$ is the conjugate transpose, we conclude that $F$ is a multiple of
\begin{equation*}
\begin{pmatrix}
i & 0 \\ 0 & 0
\end{pmatrix}.
\end{equation*}

\item
$\End_{M_2(\Q(i))}(V) \cong M_1(\Q(i)) \cong \Q(i)$.  It follows that $\G_V \cong \U(1;\Q(i))$ and $\G_V(\Z) \cong \U(1;\Z[i])$.
\end{enumerate}

\para{Case 3.al}
In this case $G = S_4$.  The representation is the sign representation tensored with the standard representation.  This case proceeds precisely as case 3.z, and we conclude that $\G_V \cong \SL_2(\Q)$ and $\G_V(\Z) \cong \SL_2(\Z)$.

\subsection{Direct proofs of arithmeticity}

We now give direct proofs for cases 2.b, 3.b, 3.e, 3.s.1, and 3.ad.2.  For cases 2.b, 3.b, and 3.e, we directly find elements of $\Im(\Phi_V)$ that generate a finite index subgroup of $\mathcal{G}_V(\Z)$.  The methods of Section 4 also apply to cases 2.b and 3.e, but these actions are simple enough that we can understand them directly; the methods of Section 4 do not apply to 3.b because the codomain of $\Phi_V$ is $\Sp(4,\Z)$ in this case.  For cases 3.s.1 and 3.ad.2, we use the results of Section 3.2 to conclude that $\Im(\Phi_V)$ is finite.

Throughout these proofs, for a simple closed curve $c$, let $T_c$ denote the left Dehn twist around $c$.

\para{Case 2.b}
\begin{figure}[h]
\center
\includegraphics[scale=.2]{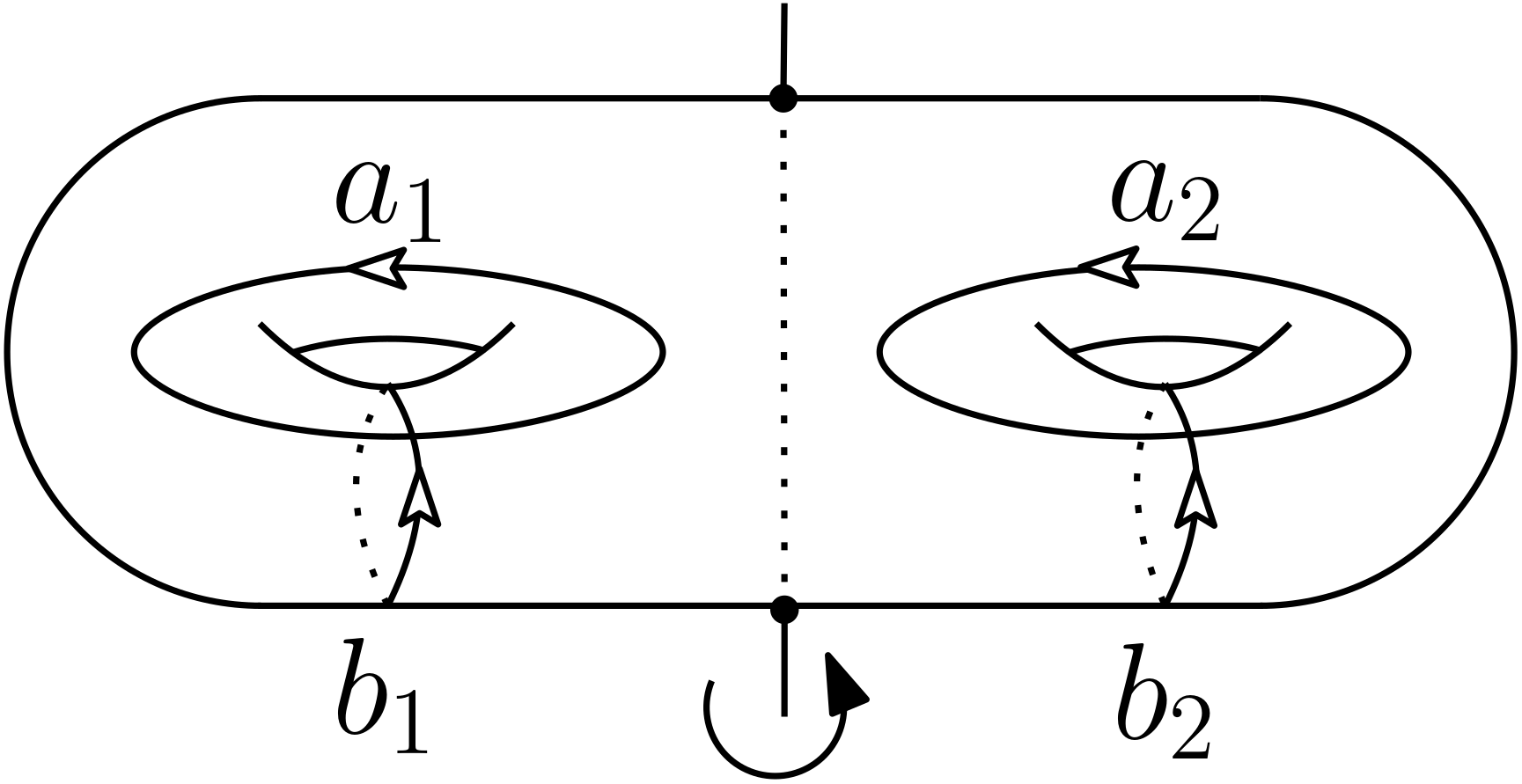}
\caption{The action of $C_2$ in Case 2.b.}
\label{fig:2-b}
\end{figure}

In this case, $G = C_2$ and $S$ is a genus 2 surface.  The action of $G$ on $S$ is the order 2 rotation pictured in Figure \ref{fig:2-b}.  Let $a_i$, $b_i$ be the standard basis of $H_1(S;\Q)$.  We can see directly that
\begin{equation*}
H_1(S;\Q) = \Q\{a_1 - a_2, b_1-b_2\} \oplus \Q_-\{a_1 + a_2, b_1+b_2\}
\end{equation*}
where $\Q$ denotes the trivial representation and $\Q_-$ denotes the sign representation.  Observe that the Reidemeister pairing on $H_1(S;\Q)_{\Q_-}$ is simply $2\hat{i}$, and hence with respect to the above basis, $\G_{\Q_-} = \SL_2(\Q)$ and $\G_{\Q_-}(\Z) = \SL(2,\Z)$.  With respect to this basis, the map $\Mod(S)^{C_2} \rightarrow \G_{\Q_-}$ maps
\begin{align*}
T_{a_1}T_{a_2} &\mapsto 
\begin{pmatrix}
1 & -1 \\ 0 & 1
\end{pmatrix} \\
T_{b_1}T_{b_2} &\mapsto
\begin{pmatrix}
1 & 0 \\ 1 & 1
\end{pmatrix}.
\end{align*}
Thus, we generate $\SL_2(\Z)$. In particular, $\Im(\Phi_{\Q_-})$ is arithmetic.

\para{Case 3.b}
\begin{figure}[h]
\center
\includegraphics[scale=.2]{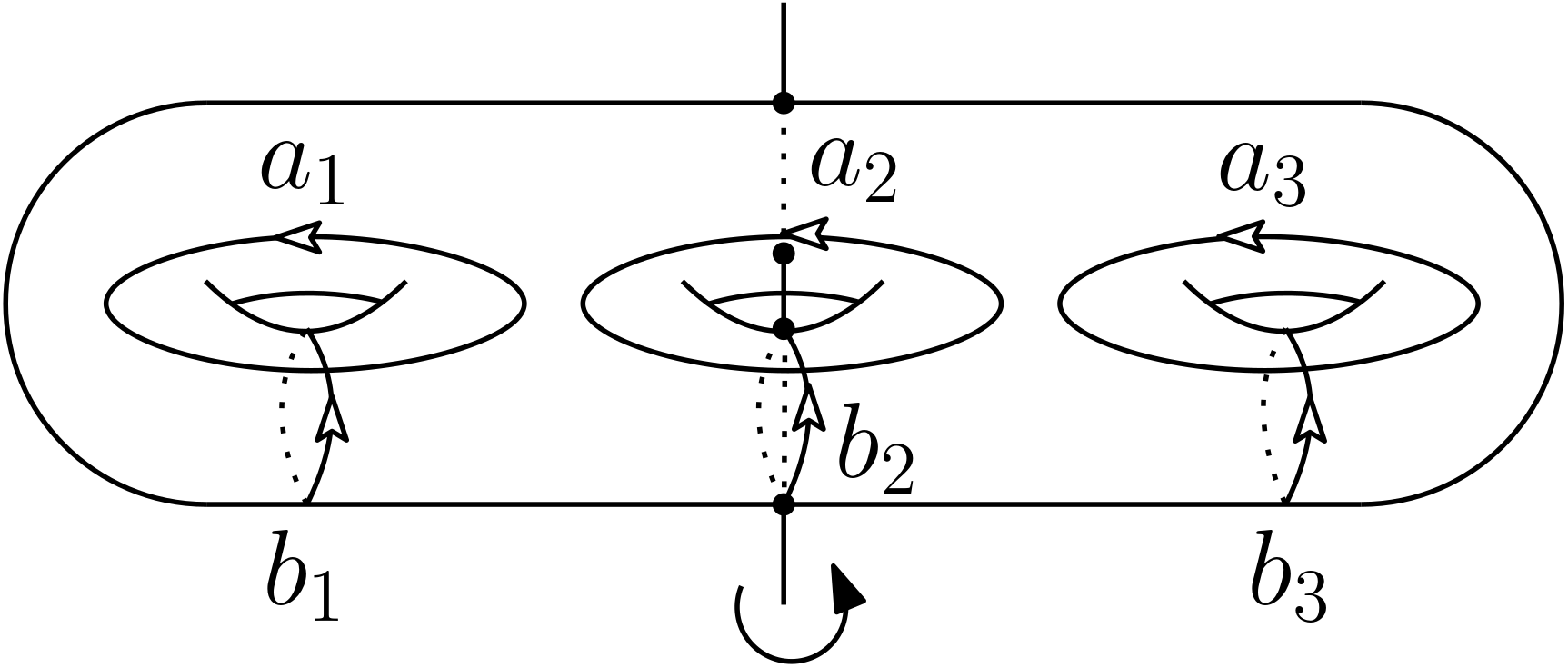}
\caption{The action of $C_2$ in Case 3.b.}
\label{fig:3-b}
\end{figure}

In this case, $G = C_2$ and $S$ is a genus 3 surface.  The action of $G$ on $S$ is the order 2 rotation pictured in Figure \ref{fig:3-b}.  Let $a_i$, $b_i$ be the standard basis of $H_1(S;\Q)$.  We can see directly that 
\begin{equation*}
H_1(S;\Q) = \Q\{a_1 - a_3, b_1 - b_3\} \oplus \Q_-\{a_1 + a_3, b_1 + b_3, a_2, b_2\}
\end{equation*}
where $\Q$ denotes the trivial representation and $\Q_-$ denotes the sign representation.
Let $A_1 = a_1 + a_3$, $B_1 = b_1 + b_3$, $A_2 = a_2$, $B_2 = 2b_2$, and let $W = \Z\{A_1,B_1,A_2,B_2\}$.  Note that $\G_{\Q_-}(\Z)$ is commensurable to $\Aut_\Z(W, \eta)$.  Again the Reidemeister pairing on $H_1(S;\Q)_{\Q_-}$ is $2\hat{i}$.  If we let $\omega$ be the form on $W$ defined in our basis by
\begin{equation*}
\omega = \begin{pmatrix}
0 & 1 & 0 & 0 \\ -1 & 0 & 0 & 0 \\ 0 & 0 & 0 & 1 \\ 0 & 0 & -1& 0
\end{pmatrix},
\end{equation*}
then $\hat{i} = 2\omega$ and $\eta = 4\omega$ on $W$.  So, $\Aut(W, \eta) = \Sp(4,\Z)$.  Given $v \in W$, let $R_v$ denote the matrix of the transvection
\begin{equation*}
x \mapsto x + \omega(x,v)v
\end{equation*}
in our basis of $W$.
The level 2 principal congruence subgroup of $\Sp(4,\Z)$ is generated by the following set (see e.g.\ \cite{Tit}):
\begin{equation*}
\{R_{A_1}^2, R_{A_2}^2, R_{B_1}^2, R_{B_2}^2, R_{A_1+A_2}^2, R_{B_1+B_2}^2, R_{A_1+B_1}^2, R_{A_1+B_2}^2, R_{A_2+B_1}^2, R_{A_2+B_2}^2\}.
\end{equation*}
So, it suffices to show that the image of our representation contains all of the above elements.

First, note that
\begin{align*}
T_{a_1}T_{a_3} &\mapsto R_{A_1} \\
T_{a_2} & \mapsto R_{A_2}^2 \\
T_{b_1}T_{b_3} &\mapsto R_{B_1} \\
T_{b_2} &\mapsto R_{B_2}^2
\end{align*}
Observe that $R_{A_1}$ and $R_{B_1}$ together generate $\SL_2(\Z) \times \{I_2\}$, and $R_{A_2}^2$ and $R_{B_2}^2$ generate $\{I_2\} \times \left\la \left( \begin{smallmatrix} 1 & 2 \\ 0 & 1 \end{smallmatrix} \right), \left( \begin{smallmatrix} 1 & 0 \\ 2 & 1 \end{smallmatrix} \right) \right\ra$.  Moreover, the generator of $C_2$ gives $-I_4$ in the image, hence  $-I_4(-I_2 \oplus I_2) = I_2 \oplus -I_2$ is also the image, where $\oplus$ denotes a block diagonal matrix.  Thus, the image contains $\SL_2(\Z) \times \Gamma(2)$, where $\Gamma(2)$ is the level 2 principal congruence subgroup of $\SL_2(\Z)$.  This gets us $R_{A_1+B_1}^2$ and $R_{A_2+B_2}^2$.

So, it remains to get the ``mixed" transvections $R_{A_1+A_2}^2$, $R_{A_1+B_2}^2$, $R_{B_1+A_2}^2$, and $R_{B_1+B_2}^2$.  For a homology class $\gamma \in H_1(S;\Z)$, let $T_\gamma$ denote any Dehn twist around a simple closed curve that represents $\gamma$.  Now, note that

\begin{align*}
T_{a_1+b_2}T_{a_3+b_2} &\mapsto 
\begin{pmatrix}
1 & -1 & 1 & 0 \\
0 & 1 & 0 & 0 \\
0 & 0 & 1 & 0 \\
0 & -1 & 1 & 1
\end{pmatrix} = R_{A_1+B_2}, \\
T_{b_1+b_2}T_{b_3+b_2} &\mapsto
\begin{pmatrix}
1 & 0 & 0 & 0 \\
1 & 1 & 1 & 0 \\
0 & 0 & 1 & 0 \\
1 & 0 & 1 & 1
\end{pmatrix} = R_{B_1+B_2},
\end{align*}
and
\begin{align*}
(T_{a_1+a_2}T_{a_3+a_2})(T_{a_1}T_{a_3})(T_{a_2}^{-1}) &\mapsto
\begin{pmatrix}
1 & -1 & 0 & -2 \\
0 & 1 & 0 & 0 \\
0 & -2 & 1 & -4 \\
0 & 0 & 0 & 1 \\
\end{pmatrix}
\begin{pmatrix}
1 & -1 & 0 & 0 \\ 0 & 1 & 0 & 0 \\ 0 & 0 & 1 & 0 \\ 0 & 0 & 0 & 1
\end{pmatrix}
\begin{pmatrix}
1 & 0 & 0 & 0 \\ 0 & 1 & 0 & 0 \\ 0 & 0 & 1 & 2 \\ 0 & 0 & 0 & 1 
\end{pmatrix} \\
&=
\begin{pmatrix}
1 & -2 & 0 & -2 \\
0 & 1 & 0 & 0 \\
0 & -2 & 1 & -2 \\
0 & 0 & 0 & 1
\end{pmatrix}
= R_{A_1+A_2}^2, \\
(T_{b_1+a_2}T_{b_3+a_2})(T_{b_1}T_{b_3})(T_{a_2}^{-1}) &\mapsto
\begin{pmatrix}
1 & 0 & 0 & 0 \\
1 & 1 & 0 & -2 \\
2 & 0 & 1 & -4 \\
0 & 0 & 0 & 1
\end{pmatrix}
\begin{pmatrix}
1 & 0 & 0 & 0 \\ 1 & 1 & 0 & 0 \\ 0 & 0 & 1 & 0 \\ 0 & 0 & 0 & 1
\end{pmatrix}
\begin{pmatrix}
1 & 0 & 0 & 0 \\ 0 & 1 & 0 & 0 \\ 0 & 0 & 1 & 2 \\ 0 & 0 & 0 & 1 
\end{pmatrix} \\
&= 
\begin{pmatrix}
1 & 0 & 0 & 0 \\
2 & 1 & 0 & -2 \\
2 & 0 & 1 & -2 \\
0 & 0 & 0 & 1 
\end{pmatrix}
= R_{B_1+A_2}^2.
\end{align*}
So, it suffices to check that the multitwists $T_{a_1+b_2}T_{a_3+b_2}$, $T_{b_1+b_2}T_{b_3+b_2}$, $T_{a_1+a_2}T_{a_3+a_2}$, and $T_{b_1+a_2}T_{b_3+a_2}$ lie in $\Mod(S)^{C_2}$, which is to say that these homology classes are represented by multicurves which are invariant under $C_2$.  We can simply draw these curves; see Figure \ref{fig:3-bcurves}.
\begin{figure}[h]
\center
\includegraphics[scale=.25]{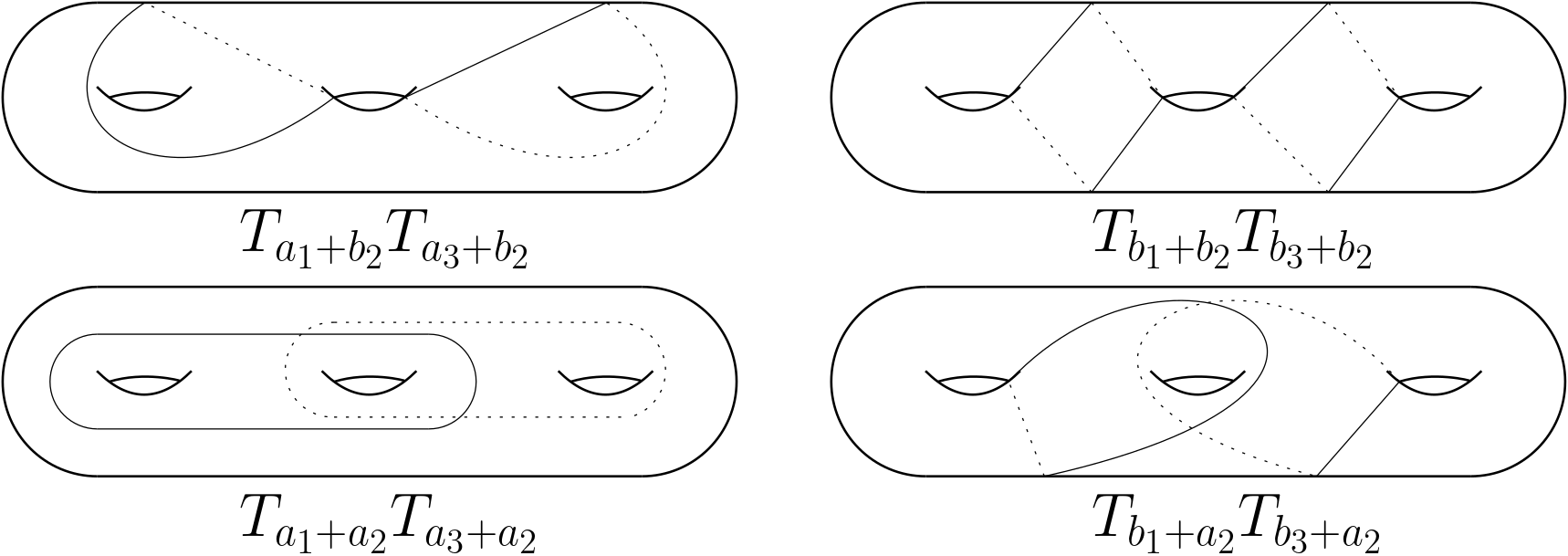}
\caption{Verifying that the desired twists lie in $\Mod(S)^{C_2}$.}
\label{fig:3-bcurves}
\end{figure}

\para{Case 3.e}
\begin{figure}[h]
\center
\includegraphics[scale=.25]{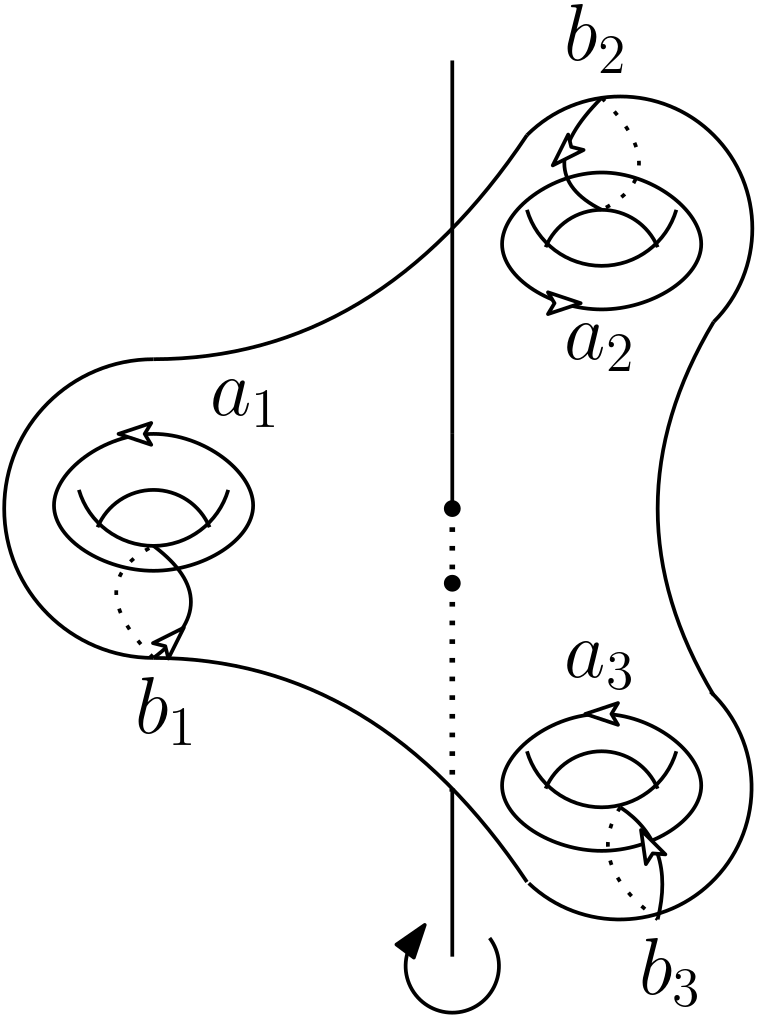}
\caption{The action of $C_3$ in Case 3.e.}
\label{fig:3-e}
\end{figure}

In this case, $G = C_3$ and $S$ has genus 3.  The action of $G$ on $S$ is the order 3 rotation pictured in Figure \ref{fig:3-e}.  Let $a_i$, $b_i$ be the standard basis of $H_1(S;\Q)$.  We see that
\begin{equation*}
H_1(S;\Q) = \Q\{a_1+a_2+a_3, b_1+b_2+b_3\} \oplus \Q(\zeta_3)\{a_1-a_2, b_1-b_2\}
\end{equation*}
where $\Q$ is the trivial representation and $\zeta$ acts by the generator of $C_3$ (in particular, $1+c+c^2$ acts by $0$ on $a_1-a_2$ and $b_1-b_2$).  With respect to this basis, we see that the Reidemeister pairing on $H_1(S;\Q)_{\Q(\zeta_3)}$ is $3h$, where $h$ is the skew-Hermitian form
\begin{equation*}
h =
\begin{pmatrix}
0 & 1 \\ -1 & 0
\end{pmatrix}.
\end{equation*}
In this case the corresponding simple factor is the field $\Q(\zeta_3)$, the involution $\tau$ is complex conjugation, and the field fixed by $\tau$ is $K = \Q(\zeta_3+\zeta_3^{-1}) = \Q$.  So, $\G_{\Q(\zeta_3)}(\mathcal{O}_K) = \G_{\Q(\zeta_3)}(\Z) = U(h;\Z[\zeta_3])$.  Since the determinant of every $B \in U(h;\Z[\zeta_3])$ is an integer on the unit circle, there are only finitely many possible determinants, and so $\SU(h;\Z) = \SL_2(\Z)$ has finite index in $\G_{\Q(\zeta_3)}(\Z)$.

With respect to the above basis, the map $\Mod(S)^{C_3} \rightarrow \G_{\Q(\zeta_3)}$ maps
\begin{align*}
T_{a_1}T_{a_2}T_{a_3} &\mapsto 
\begin{pmatrix}
1 & -1 \\ 0 & 1
\end{pmatrix} \\
T_{b_1}T_{b_2}T_{b_3} &\mapsto
\begin{pmatrix}
1 & 0 \\ 1 & 1
\end{pmatrix}.
\end{align*}
Thus, we generate the entire group $\SL_2(\Z)$. In particular, $\Im(\Phi_{\Q(\zeta_3)})$ is arithmetic.

\para{Case 3.s.1}
In this case $G = Q_8$.  Recall from Section 3.2 that for the unique faithful $V \in \Irr(G)$, $\G_V = \{B \in D^\times \mid BFB^* = F\}$ where $F \in M_2(\Q(i))$ and
\begin{equation*}
D = \left\{
\begin{pmatrix}
\alpha & \beta \\ -\overline{\beta} & \overline{\alpha}
\end{pmatrix}
\mid \alpha,\beta \in \Q(i)\right\}.
\end{equation*}
Note that $\G_V^1 := \{B \in \G_V \mid \det(B) = 1\}$ is precisely the group $\SU(2;\Q(i))$.  Since $\SU(2)$ is compact, every discrete subgroup of $\G_V^1$ is finite.  Thus, this case is automatically arithmetic if we can show that $\G_V^1(\Z)$ has finite index in $\G_V(\Z)$ (this is actually true in more general cases, see \cite[Proposition~3.9]{GLLM}).  This follows as in the previous case.  Namely, observe first that for every $B \in \G_V$, $\v \det(B) \v^2 \det(F) = \det(F)$ and hence $\v \det(B) \v = 1$.  On the other hand, for every $B \in \G_V(\Z)$, $\det(B) \in \Z[i]$.  It follows that $\det(B) \in \{\pm 1, \pm i\}$ for every $B \in \G_V(\Z)$, and hence $\G_V^1(\Z)$ has finite index as desired.  

\para{Case 3.ad.2}
In this case $G = C_4 \circ D_4$.  Recall from Section 3.2 that for the unique faithful $V \in \Irr(G)$, $\G_V \cong U(1; \Q(i))$.  Since the group $U(1)$ is compact, this case is automatically arithmetic.

%-------------------------------------------------------------------------------------------------------------------------------------------------------------------

\section{Lifting twists}

In the 12 cases that remain (2.k.2, 2.n, 2.s, 3.i.1, 3.m, 3.n, 3.q.2, 3.r.2, 3.s.2, 3.y, 3.z, and 3.al), the group $\G_V$ is isomorphic to $\SL_2(\Q)$ or $\U(1,1;\Q[i])$.  Since $\SU(1,1;\Q[i])$ is isomorphic to $\SL_2(\Q)$, we can always map $\G_V(\Z)$ to $\SL_2(\Z)$ (up to commensurability) in the remaining cases.  Given a set of matrices in $\SL_2(\Z)$, one can determine if they generate a finite index subgroup by computing the intersection with a finite index free subgroup and applying an algorithm for determining whether a subgroup of a free group has finite index (see e.g.\ \cite{Kap}).  Thus, if we can explicitly compute elements of $\Im(\Phi_V)$, it is simple to check if we have found enough to generate an arithmetic subgroup.

To find elements of $\Mod(S)^G$, we can simply lift elements of $\Mod(S^\circ/G)$.  However, for an arbitrary mapping class of $S^\circ/G$, it is not easy in general to compute the action of the lift on $H_1(S;\Q)$.  The simplest elements of $\Mod(S)^G$ are lifts of powers of Dehn twists, which we call \emph{partial rotations};  these elements have a simple geometric description that allows us to compute their action on $H_1(S;\Q)$.  While there is not an obvious reason to expect it a priori, it turns out that partial rotations are enough to generate an arithmetic subgroup in all 12 remaining cases.

To compute the action of a partial rotation, we equip $S$ with a cell structure for which $G$ acts by cellular maps.  However, the partial rotation will not be a cellular map.  Instead, we directly define a map on the cellular chain group $\psi:C_1(S) \rightarrow C_1(S)$ (this map is \emph{not} induced by a map $S \rightarrow S$), and we observe that $\psi$ induces a map on homology that agrees with the action of the partial rotation.

In Section 4.1, we build the cell structure on $S$ and explain how to compute the action of a partial rotation on $H_1(S;\Q)$.  In Section 4.2, we explain in detail how to determine whether a finitely generated subgroup of $\SL_2(\Z)$ has finite index.  In Section 4.3, we present a list of partial rotations that generate a finite index subgroup of $\G_V(\Z)$ for each of the remaining 12 cases.  We carry out these computations in SageMath; the code for these computations is available at \cite{low-genus-actions}.

\subsection{Computing the action of lifted twists}

\para{An equivariant cell structure}
Fix a cover $p:S \rightarrow S/G$ with monodromy $\varphi:\pi_1(S^\circ/G) \rightarrow G$, and suppose $S/G$ has genus $h$ and $n > 0$ branch points $x_1, \ldots, x_n$.  Let $N = 2h+n - 1$, so $\pi_1(S^\circ/G)$ is free of rank $N$.  

Our first step is to build a $2N$-gon $P$ with oriented edges $e_1, \ldots, e_{2N}$ and orientation-preserving edge pairings so that if we let $\overline{P}$ be the quotient of $P$ by the edge pairings, then
\begin{enumerate}
\item $\overline{P}$ is homeomorphic to $S/G$, and
\item the vertices of $\overline{P}$ map to the branch points of $S/G$.
\end{enumerate}
We do this is as follows.  First, for $1 \leq i \leq h$, choose simple loops $c_i$ and $d_i$ based at $x_n$ representing the standard generators of $\pi_1(S/G)$.  Next, for $1 \leq i \leq n-1$, choose arcs $\alpha_i$ from $x_n$ to $x_i$.  Assuming our loops and arcs intersect only at $x_n$, we can cut along them to obtain the desired polygon $P$; see Figure \ref{fig:polygonP}.

\begin{figure}[h]
\center
\includegraphics[scale=.25]{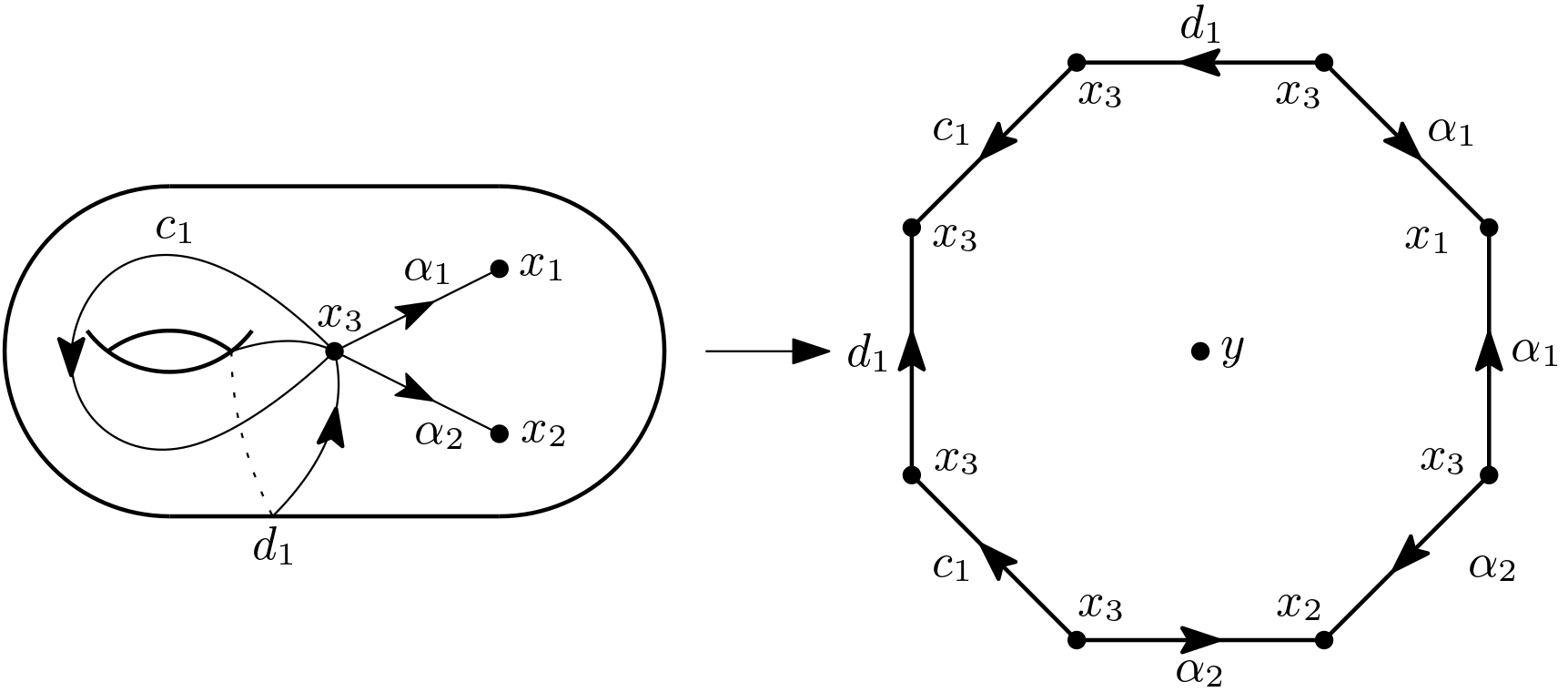}
\caption{Building the polygon $P$, with $h=1$ and $n=3$.}
\label{fig:polygonP}
\end{figure}

The edge pairings partition the edges into pairs $E_1, \ldots, E_N$.  For $1 \leq i \leq N$, write $E_i = \{e_{i_1}, e_{i_2}\}$ where $i_1 < i_2$, and set $E_{i}^+ = e_{i_1}$ and $E_i^- = e_{i_2}$.  We fix the center point $y \in P$ as the base point of $S^\circ/G$.  For each edge pair $E_i$, we build a loop $\gamma_i$ on $S^\circ/G$ based at $y$ by choosing an arc from $y$ to edge $E_i^+$, and then an arc out of edge $E_i^-$ back to $y$.  Together, the loops $\gamma_i$ for $1 \leq i \leq N$ form a free basis of $\pi_1(S^\circ/G)$.

Now, we construct the cell structure on $S$.  Take $\v G \v$ copies of the polygon $P$, labeled $P_u$ for $u \in G$; we call these the \emph{sheets} of $S$.  We let $e_j^u$ denote the copy of edge $e_j$ on the sheet $P_u$.  Our goal is to glue the edges of the sheets $P_u$ so that the permutation action $v \cdot P_u = P_{vu}$ is precisely the action of $G$ on $S$.  Let $y_u$ be the copy of $y$ on $P_u$; we fix $y_{id}$ as the base point on $S$.  To get the desired gluing, we need the lift of $\gamma_i$ to $y_u$ to end at $y_{u\varphi(\gamma_i)}$.  To achieve this, we simply glue $(E_i^+)^u$ to $(E_i^-)^{u\varphi(\gamma_{E_i})}$ for each edge pair $E_i$.  See Figure \ref{fig:exampleCover} for an illustration in Case 2.k.2.

\begin{figure}[h]
\center
\includegraphics[scale=.25]{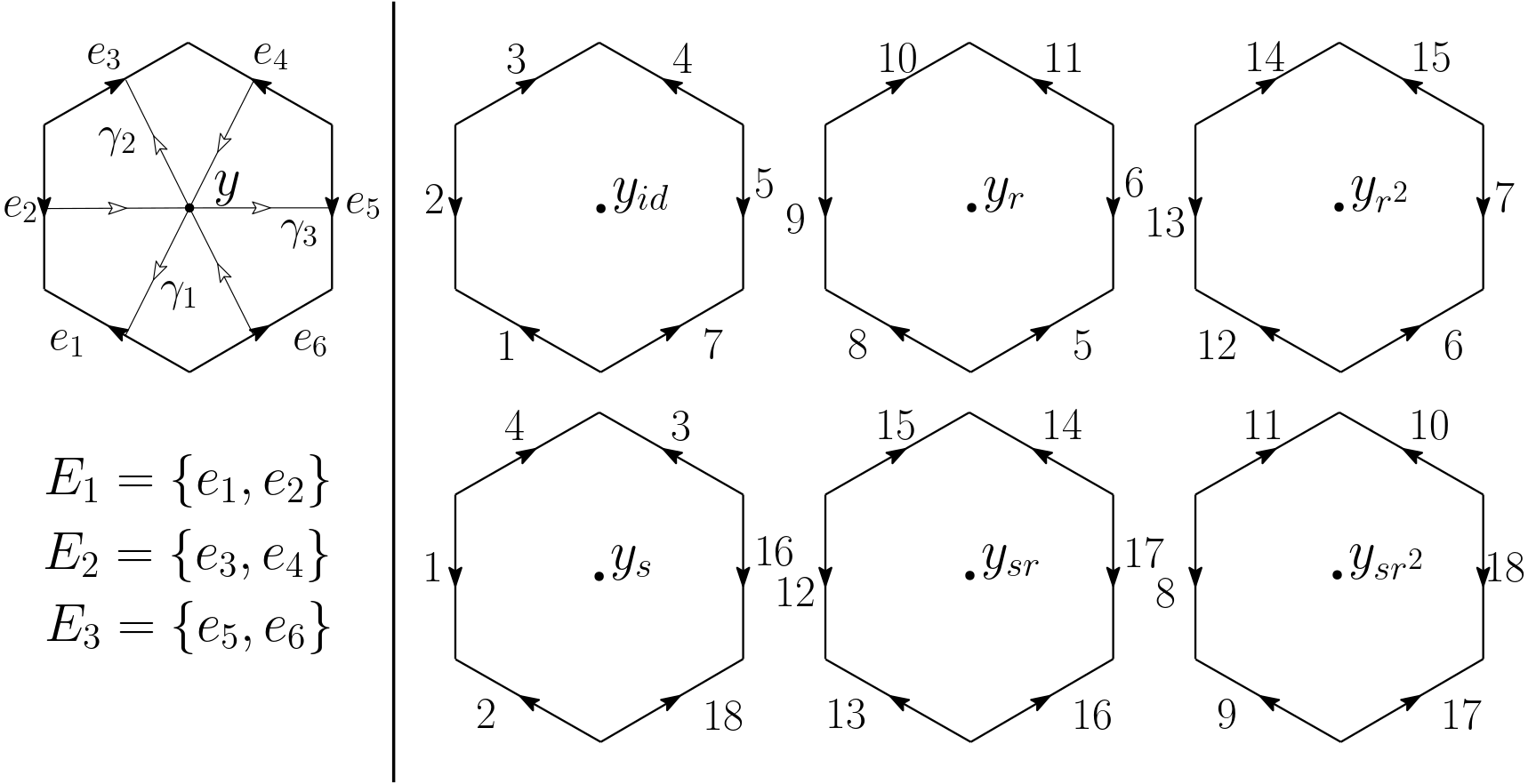}
\caption{Case 2.k.2. Here $G= D_3$, $g=2$, $h=0$, $n= 4$, $\varphi(\gamma_1) = s$, $\varphi(\gamma_2) = s$, and $\varphi(\gamma_3) = r$.  The edge identifications in the cover are given by numbers.}
\label{fig:exampleCover}
\end{figure}

Thus, we have a cell structure on $S$ so that $G$ acts by cellular maps.  In particular, the action of $G$ on $H_1(S;\Q)$ can be computed directly from the permutation action on the edges.

\para{Partial rotations}
Fix an oriented simple closed curve $\delta$ on $\overline{P} = S^\circ/G$.  We assume the following about $\delta$:
\begin{itemize}
\item $\delta$ is transverse to the edges of $P$,
\item $\delta$ does not pass through any vertices of $P$,
\item the base point $y \in P$ lies to the left of $\delta$ (in particular, $\delta$ does not pass through $y$).
\end{itemize}
We also fix a simple arc in $\Int(P) \backslash \delta$ from $y$ to a point $y' \in \delta$; in this way we can view $\delta$ as a well-defined element of $\pi_1(S^\circ/G)$.  If we let $y'_{id}$ denote the copy of $y'$ on sheet $P_{id}$ and let $\widetilde{\delta}_{id}$ be the lift of $\delta$ to $y'_{id}$, then $\varphi(\delta)$ is the deck transformation taking $y'_{id}$ to the endpoint of $\widetilde{\delta}_{id}$.  

Let $T_\delta$ denote the left Dehn twist around $\delta$; we may assume that $T_\delta$ fixes the base point $y \in P$.  Assume that $T_\delta$ lifts to a homeomorphism $\widetilde{T}_\delta$ of $S$ which commutes with $G$; recall that this occurs if and only if $T_\delta$ fixes the monodromy homomorphism $\varphi:\pi_1(S^\circ/G) \rightarrow G$.  In particular we assume that $\widetilde{T}_\delta$ is the lift fixing $y_{id}$.  The lift $\widetilde{T}_\delta$ is a \emph{partial rotation}.  This is a homeomorphism that looks like a fractional Dehn twist around each component of $p^{-1}(\delta)$.  We make this precise as follows.

First, observe that $G$ acts on the set of components of $p^{-1}(\delta)$, and the subgroup $\la \varphi(\delta) \ra$ is the stabilizer of the component passing through $y'_{id}$, which we denote $C_{id}$.  Thus we get an isomorphism of $G$-sets
\begin{equation*}
\{\text{components of $p^{-1}(\delta)$}\} \cong G/\la \varphi(\delta) \ra,
\end{equation*}
where $C_{id}$ corresponds to the coset $\la \varphi(\delta) \ra$.  In particular, if we let $m$ denote the order of $\la \varphi(\delta) \ra$ in $G$, then $p^{-1}(\delta)$ has $\v G \v/m$ components, each containing $m$ lifts of $\delta$.  

Now, we can precisely examine the action of $\widetilde{T}_\delta$. Let $U \subseteq S$ be the preimage of the support of $T_\delta$ (so $U$ is an annular neighborhood of each component of $p^{-1}(\delta)$), and let $S^\times = S \backslash U$.  On $U$, the partial rotation $\widetilde{T}_\delta$ must act by a $(1/m)$th rotation on each component.  On $S^\times$, $\widetilde{T}_\delta$ must act by a deck transformation on each component.  To understand which deck transformations it acts by, we first look at the action around $C_{id}$.  Since $\widetilde{T}_\delta$ fixes $y_{id}$ and $y_{id}$ lies to the left of $C_{id}$, the lift $\widetilde{T}_\delta$ in fact fixes the entire component of $S^\times$ to the left of $C_{id}$.  It then must act on the component to the right of $C_{id}$ by $\varphi(\delta)$.  More generally, for every component $C$ of $p^{-1}(\delta)$, if we let $u\la \varphi(\delta)\ra$ be the corresponding coset, then $\widetilde{T}_\delta$ fixes the component of $S^\times$ to the left of $C$ and acts on the component of $S^\times$ to the right of $C$ by $u\varphi(\delta)u^{-1}$.

\para{The action on homology}
Computing the action of the partial rotation $\widetilde{T}_\delta$ on $H_1(S;\Q)$ is not immediate as it is not a cellular map; it acts on the set of vertices but does not act on the set of edges.  Instead, we define a map on the cellular chain group $\psi:C_1(S) \rightarrow C_1(S)$ which induces a map on $H_1(S;\Q)$ that agrees with the action of $\widetilde{T}_\delta$.  To define $\psi$, we examine the arc $\widetilde{T}_\delta(e)$ for each edge $e$ on $S$, and we explain a procedure, independent of $e$, to homotope $\widetilde{T}_\delta(e)$ rel endpoints into the 1-skeleton of $S$.  This gives us a consistent way, for each edge $e$, to write $\widetilde{T}_\delta(e)$ as a signed sum of edges; we define $\psi(e)$ to be this signed sum.  Note that $\psi$ is not induced by a map $S \rightarrow S$; in particular, we are not simply homotoping $\widetilde{T}_\delta$ to a cellular map.  However, $\psi$ still induces a map on homology, as we check in Lemma \ref{lem:mapOnH1} below.

We can understand the arc $\widetilde{T}_\delta(e)$ informally using our geometric description of $\widetilde{T}_\delta$.  To draw $\widetilde{T}_\delta(e)$, we start at the initial vertex of $e$ and head towards the terminal vertex.  Every time we run into $p^{-1}(\delta)$, we follow it to the left for one lift of $\delta$, and then continue on a copy of $e$ on a different sheet.  We continue this process until we reach the terminal vertex of some copy of $e$.  If $\widetilde{T}_\delta$ fixes the component of $S^\times$ containing the intial vertex of $e$, then we are done.  If it acts on this component by $u \in G$, then the arc we have currently drawn is $u^{-1}\widetilde{T}_\delta(e)$, so we finish by applying the deck transformation $u$ to our arc.

Now, we can adapt this geometric description to define $\psi$.  Fix an edge $e_j^u$ on $S$.  Let $\delta \cap e_j = \{z_1, \ldots, z_k\}$, ordered by distance from the intial vertex of $e_j$ (it is possible that this set is empty).  Then for each $v \in G$, we can write
\begin{equation*}
p^{-1}(\delta) \cap e_j^v = \{z_1^v, \ldots, z_k^v\}.
\end{equation*}
First, we let $\widetilde{\delta}_1$ be the lift of either $\delta$ to $\delta^{-1}$ to $z_1^u$; we choose $\delta$ if $\hat{i}(e_j,\delta)$ is positive at $z_1$, and choose $\delta^{-1}$ otherwise (i.e.\ we want our lift to follow $p^{-1}(\delta)$ ``to the left").  Inductively, for $1 \leq i < k$, let $u_i \in G$ be the element such that $\widetilde{\delta}_i$ ends at $z_i^{u_i} \in e_j^{u_i}$, and we let $\widetilde{\delta}_{i+1}$ be the lift of either $\delta$ or $\delta^{-1}$ to $z_{i+1}^{u_i}$ depending on whether $\hat{i}(e_j, \delta)$ is positive or negative at $z_{i+1}$.  The arc $\widetilde{T}_\delta(e_j^u)$ is then obtained by joining the intial vertex of $e_j^u$ to the initial vertex of $\widetilde{\delta}_1$, joining the terminal vertex of $\widetilde{\delta}_i$ to the initial vertex of $\widetilde{\delta}_{i+1}$ for $1 \leq i < k$, and joining the terminal vertex of $\widetilde{\delta}_k$ to the terminal vertex of $e_j^{u_k}$.  See Figure \ref{fig:exampleLifts} for an illustration of $\widetilde{T}^2_\delta(e_j^u)$ in Case 2.k.2 (for a $k$th power of a twist, we take $\widetilde{\delta}_i$ to be a lift of $\delta^{\pm k}$).

\begin{figure}[h]
\center
\includegraphics[scale=.25]{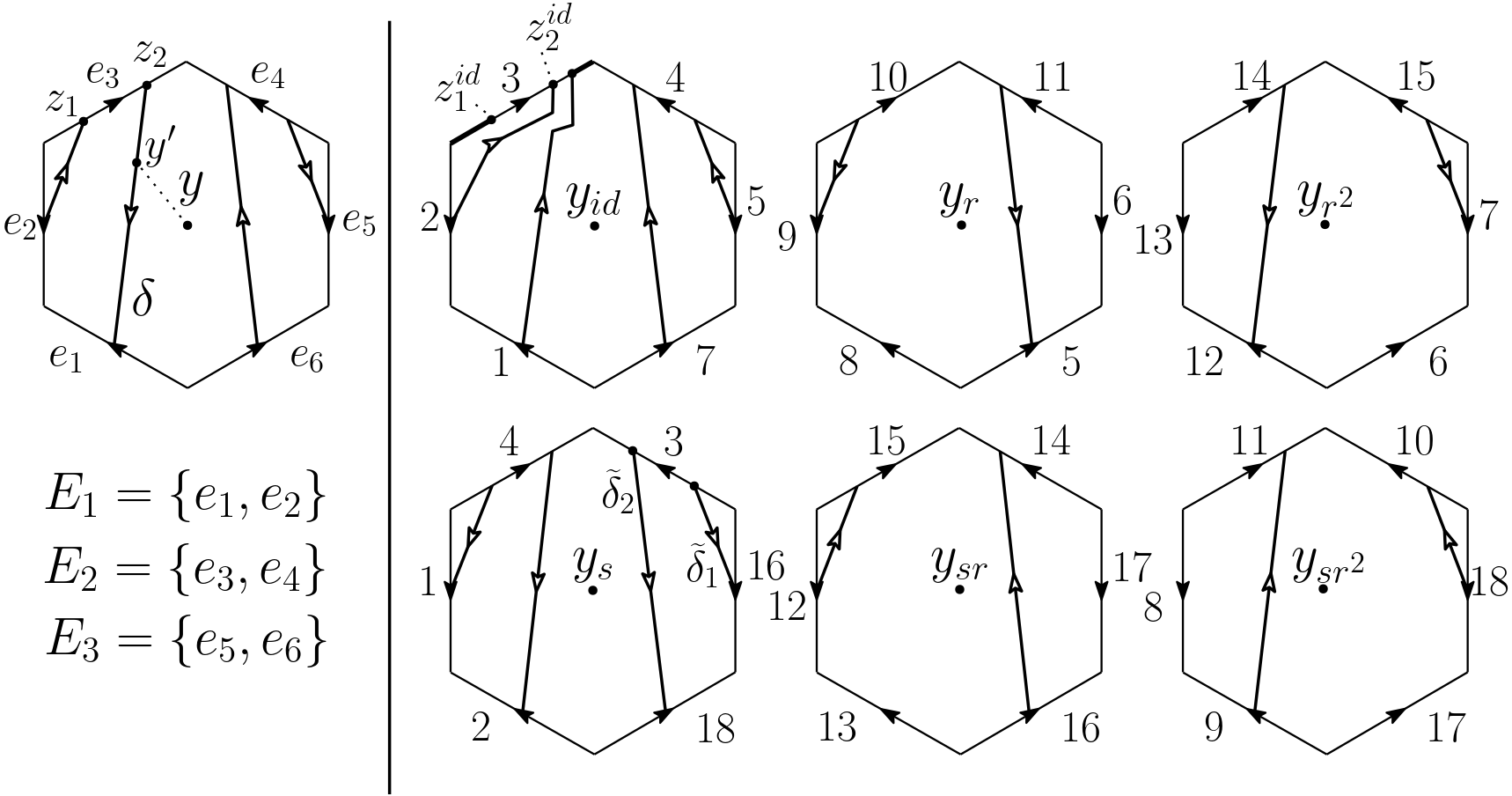}
\caption{Example of a curve $\delta$ and the arc $\widetilde{T}^2_\delta(e_3^{id})$ in Case 2.k.2.  Here $\widetilde{\delta}_1$ is a lift of $\delta^2$, and $\widetilde{\delta}_2$ is a lift of $\delta^{-2}$.  Both lifts happen to end on the same sheet where they started.  Note that some overlap occurs on edge $e_3^{id}$; in the illustration we slightly perturb the arcs for clarity.}
\label{fig:exampleLifts}
\end{figure}

Now, we homotope each $\widetilde{\delta}_i$ into the 1-skeleton of $S$ in two stages:
\begin{enumerate}
\item For each edge $e$ that $\widetilde{\delta}_i$ passes through, we push $\widetilde{\delta}_i$ down $e$ so that it passes through the initial vertex of $e$.  Now on each sheet, $\widetilde{\delta}_i$ is union of arcs whose endpoints are vertices.
\item On each sheet, we push each arc of $\widetilde{\delta}_i$ onto the boundary so that it traverses the boundary clockwise.
\end{enumerate}
These homotoped arcs, together with the full edge $e_j^{u_{k+1}}$, are homotopic rel endpoints to the arc $\widetilde{T}_\delta(e_j^u)$. We illustrate these steps in Figure \ref{fig:exampleHomotope}.

\begin{figure}[h!]
\center
\includegraphics[scale=.48]{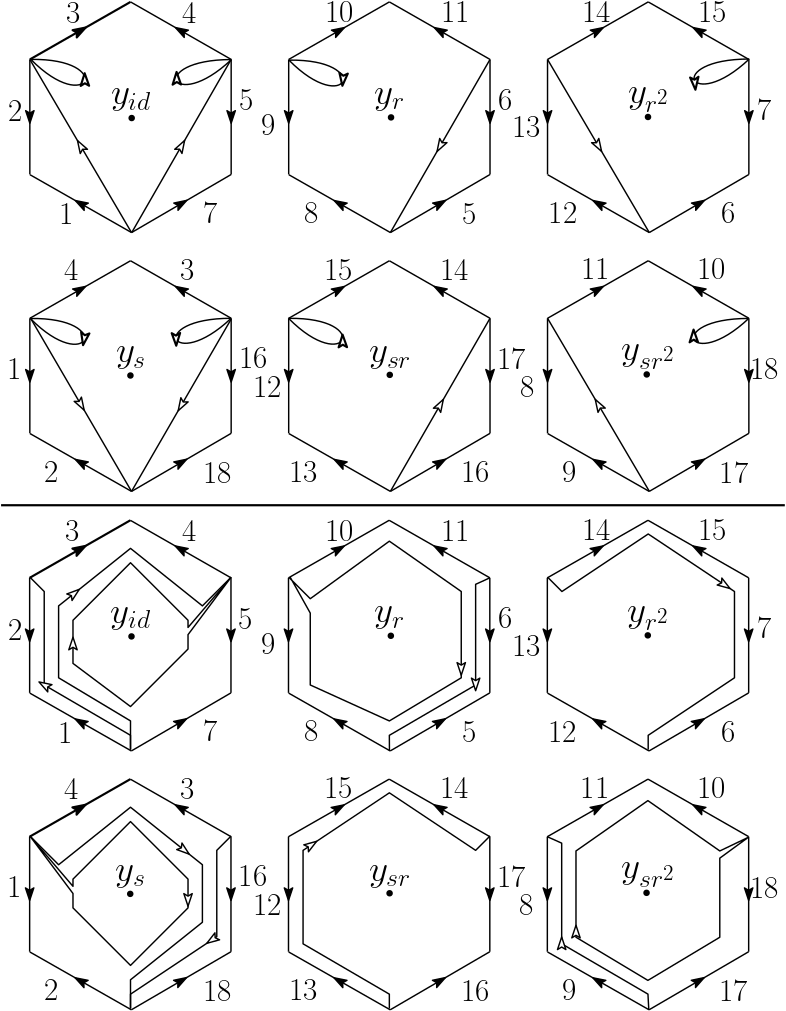}
\caption{Implementing steps 1 and 2 of the homotopy in Case 2.k.2.  After the second step, the arcs lie on the 1-skeleton, but we draw them parallel to the 1-skeleton for clarity.}
\label{fig:exampleHomotope}
\end{figure}

This procedure gives us a well-defined way to represent each $\widetilde{\delta}_i$ as a signed sum of edges, and hence as an element of $C_1(S)$.  In particular, if we let $b_{i\ell}^v$ denote the number of times that the homotoped version of $\widetilde{\delta}_i$ traverses the edge $e_\ell^v$, counted with sign depending on whether the orientations agree, then $\widetilde{\delta}_i$ is represented by the element
\begin{equation*}
\sum_{v \in G} \sum_{\ell=1}^{2N} b_{i\ell}^v e_\ell^v \in C_1(S).
\end{equation*} 
Finally, recall that $\widetilde{T}_\delta$ acts on the component of $S^\times$ containing the initial vertex of $e_j^u$ by some deck transformation $w \in G$.  As explained above, $w$ will either be the identity or a conjugate of $\varphi(\delta)$, depending on whether the initial vertex of $e_j^u$ lies to the left or right of $p^{-1}(\delta)$.  In either case, we define
\begin{equation*}
\psi(e_j^u) = w(\widetilde{\delta}_1 + \cdots + \widetilde{\delta}_k + e_j^{u_{k+1}}).
\end{equation*}

Now, we can verify that $\psi$ gives us the desired map on homology.

\begin{lemma}\label{lem:mapOnH1}
The map $\psi:C_1(S) \rightarrow C_1(S)$ induces a map on $H_1(S;\Q)$ that agrees with the map induced by $\widetilde{T}_\delta$.
\end{lemma}
\begin{proof}
Let $\del_1:C_1(S) \rightarrow C_0(S)$ and $\del_2:C_2(S) \rightarrow C_1(S)$ be the boundary maps.  First, we check that $\psi$ preserves $\Ker(\del_1)$.  This follows since for every edge $e$, $\del_1(\psi(e))$ is the (signed) endpoints of $\widetilde{T}_\delta(e)$, so $\del_1 \circ \psi = (\widetilde{T}_\delta)_* \circ \del_1$.  Next, we check that $\psi$ preserves $\Im(\del_2)$.  Note that $\Im(\del_2)$ is generated by the face relations $\del_2 P_u = \sum_{1 \leq j \leq 2N} \ep_j e_j^u$ for each $u \in G$, where $\ep_j = +1$ if $e_j^u$ is oriented counterclockwise and $\ep_j = -e_j^u$ if $e_j^u$ is oriented clockwise.  Let $\del P_u$ be the corresponding simple closed curve on $S$. Then $\widetilde{T}_\delta(\del P_u)$ is a null-homotopic simple closed curve, comprised of arcs between the vertices of $S$.  Since we can homotope $\widetilde{T}_\delta(\del P_u)$ rel the vertices into the 1-skeleton to obtain the element $\psi(\del_2 P_u) \in C_1(S)$, we conclude that $\psi$ preserves $\Im(\del_2)$.  Thus, $\psi$ indeed induces a map on $H_1(S;\Q)$.  The fact that the induced map agrees with $\widetilde{T}_\delta$ follows by construction.
\end{proof}

As noted above, we can easily modify the definition of $\psi$ to compute the action of a lift of $T_\delta^k$.  Namely, take $\widetilde{\delta}_i$ to be a lift of $\delta^{\pm k}$, rather than just a lift of $\delta^{\pm 1}$.

\para{Lifting curves combinatorially} Note that the curve $\delta$ can be represented by a sequence of oriented simple arcs $\beta_1, \ldots, \beta_\ell$ joining the sides of $P$.  The definition of $\psi$ depends only on the following data:
\begin{enumerate}
\item the order of the $\beta_i$,
\item the edge that each $\beta_i$ enters,
\item for each edge $e$ of $P$, the order that the $\beta_i$ intersect $e$ from the initial vertex to the terminal vertex.
\end{enumerate}
So we can represent each $\beta_i$ as a pair $(e_i, k_i)$ where $e_i$ is the edge that $\beta_i$ enters and $k_i$ is an integer representing the order of $\beta_i \cap e_i$ in $\delta \cap e_i$, and we can represent $\delta$ as the ordered list $((e_1,k_1), \ldots, (e_\ell, k_\ell))$.

We can compute lifts of $\delta^{\pm 1}$ from this purely combinatorial data as follows.   Suppose $\beta_{i}$ enters the edge $e_{i}$ at the point $z$.  Let $z^u$ be the copy of $z$ on the sheet $P_u$.  Assume first that $\hat{i}(e_i,\beta_i)$ is positive at $z$.  Then to lift $\delta$ to $z_u$, we simply lift each arc in the order $\beta_{i+1}, \beta_{i+2}, \ldots, \beta_\ell, \beta_1, \beta_2, \ldots, \beta_i$.  In particular, we lift $\beta_{i+1}$ to the sheet glued to edge $e_i^u$, lift $\beta_{i+2}$ to the sheet glued to the endpoint of the lift of $\beta_{i+1}$, etc.  If $\hat{i}(e_i, \beta_i)$ is negative at $z$, then we lift $\delta^{-1}$ in the analogous way, except we lift the arcs in the order $\beta_{i}, \beta_{i-1}, \ldots, \beta_1, \beta_\ell, \beta_{\ell-1}, \ldots, \beta_{i+1}$.   In either case, the lifted arcs are again determined by the edge they enter and order that intersect that edge.  To represent the lift of $\delta^{\pm 1}$ as an element of $C_1(S)$, we simply take, for each lifted arc $\widetilde{\beta}_i$, the clockwise sequence of edges between the initial vertices of the edge where $\widetilde{\beta}_i$ starts and the edge where $\widetilde{\beta}_i$ ends.

\subsection{Checking finite index}
Now, given a finitely generated subgroup $\Delta  \subseteq \PSL_2(\Z)$, we outline a procedure to determine whether $\Delta$ has finite index (and hence we get such a procedure for $\SL_2(\Z)$).  The idea is to exploit the fact that the the level 2 principal congruence subgroup $\Gamma(2)$ is freely generated by the matrices  $\left( \begin{smallmatrix} 1 & 2 \\ 0 & 1 \end{smallmatrix}\right)$ and $\left( \begin{smallmatrix} 1 & 0 \\ 2 & 1 \end{smallmatrix}\right)$.  Our procedure goes as follows:
\begin{enumerate}
\item Compute a generating set of $\Delta \cap \Gamma(2)$ using Schreier's lemma. \\

\noindent Fix a generating set $T \subseteq \Delta$.  Let $\Delta' = \Delta \cap \Gamma(2)$, and fix a right transversal $R$ of $\Delta'$ in $\Delta$, i.e.\ a set of representatives of the right cosets $\Delta' \backslash \Delta$.  Given $\delta \in \Delta$, let $\overline{\delta}$ denote the fixed representative of the coset containing $\delta$.  Then Schreier's lemma says that $\Delta'$ is generated by the set
\begin{equation*}
\{ rt\left(\overline{rt}\right)^{-1} \mid r \in R, t \in T\}.
\end{equation*}
If we let $\overline{\Delta}$ denote the image of $\Delta$ in $\PSL(2,\Z/2\Z)$, then finding the transversal $R$ amounts to finding one element in each fiber over $\overline{\Delta}$.  In particular this can be done by simply iterating over all words in the generating set $T$. \\

\item Express each generator of $\Delta \cap \Gamma(2)$ as a word in $\left( \begin{smallmatrix} 1 & 2 \\ 0 & 1 \end{smallmatrix}\right)$ and $\left( \begin{smallmatrix} 1 & 0 \\ 2 & 1 \end{smallmatrix}\right)$ using the algorithm given by Chorna, Geller, and Shpilrain in \cite{CGS}. \\

\noindent Given a matrix $M = (m_{ij}) \in \Gamma(2)$, we call $\sum_{ij} \v m_{ij} \v$ the \emph{size} of $M$.  In \cite[Theorem~4]{CGS}, the authors prove that one can always reduce the size of $M$ by left or right multiplication by $\left( \begin{smallmatrix} 1 & 2 \\ 0 & 1 \end{smallmatrix}\right)^{\pm 1}$ or $\left( \begin{smallmatrix} 1 & 0 \\ 2 & 1 \end{smallmatrix}\right)^{\pm 1}$.  Repeating this process inductively, it will eventually terminate at the identity matrix, presenting $M$ as a word in our free generators. \\

\item Compute the index of $\Delta \cap \Gamma(2)$ in $\Gamma(2)$.  \\

\noindent Once our generators of $\Delta \cap \Gamma(2)$ are written as words in $\left( \begin{smallmatrix} 1 & 2 \\ 0 & 1 \end{smallmatrix}\right)$ and $\left( \begin{smallmatrix} 1 & 0 \\ 2 & 1 \end{smallmatrix}\right)$, we apply the Todd-Coxeter procedure to compute the index (see e.g. \cite[Ch.~5]{Holt-Bettina-OBrien}).  We perform this calculation using the implementation of the Todd-Coxeter procedure in GAP \cite{GAP4}.
\end{enumerate}

\subsection{Results}

As mentioned above, we finish the proof of Theorem \ref{thm:main} by computing the action of several partial rotations in the remaining 12 cases.  In particular, for each case we carry out the following steps:
\begin{enumerate}
\item Build the homology group $H_1(S;\Q)$ using the cell structure described above.
\item Find a set of simple closed curves on $S^\circ/G$ for which we can lift a power of each each Dehn twist to a partial rotation.
\item Compute the action of each partial rotation on $H_1(S;\Q)$ using the algorithm described above.
\item Apply a change of basis to obtain a set of matrices in $\SL_2(\Z)$.
\item Check that these generate a finite index subgroup of $\SL_2(\Z)$ using the procedure defined above.
\end{enumerate}
In Table \ref{table:results}, we provide, for each case, a set of pairs $(\delta,k)$ so that the set of partial rotations $\widetilde{T}_\delta^k$ generate a finite index subgroup of $\G_V(\Z)$.  Each such set was found experimentally.  In the table, we write each curve $\delta$ as an element of $\pi_1(S^\circ/G)$ using our standard presentation
\begin{equation*}
\la c_1, d_1, \ldots, c_h, d_h, x_1, \ldots, x_n \mid [c_1,d_1] \cdots [c_h,d_h]x_1 \cdots x_n = 1\ra.
\end{equation*}

\begin{table}
\center
\begin{tabular}{|c|c|c|c|m{3.5cm}|}
\hline
Case & $G$ & Branching data & Monodromy & Partial rotations \\
\hline
2.k.2 & $D_3$ & $(2^2,3^2)$  & $(s,s,r,r^{-1})$ & $(x_1x_2, 1)$, $(x_1x_2x_3x_2^{-1}, 2)$, $(x_2x_3,2)$ \\
\hline
2.n & $D_4$ & $(2^3,4)$  & $(s,sr,r^2,r)$ & $(x_1x_2x_3x_2^{-1}, 1)$, $(x_2x_3, 1)$, \newline $(x_1x_2,2)$ \\
\hline
2.s & $D_6$ & $(2^3,3)$  & $(s,sr,r^3,r^2)$ & $(x_1x_2x_3x_2^{-1}, 1)$, $(x_2x_3, 1)$, \newline $(x_1x_2, 3)$ \\
\hline
3.i.1 & $C_4$ & $(1;2^2)$ & $(c,1,c^2,c^2)$ & $(d_1, 1)$, \newline $(c_1^2d_1x_1, 2)$ \\
\hline
3.m & $D_3$ & $(2^4,3)$ & $(s,s,s,sr^{-1},r)$ & $(x_1x_2,1)$, $(x_1x_2x_3x_2^{-1},1)$, $(x_2x_3, 1)$, \newline $(x_1x_2x_3x_4x_3^{-1}x_2^{-1}, 3)$, $(x_2x_3x_4x_3^{-1},3)$, $(x_3x_4, 3)$ \\
\hline
3.n & $D_3$ & $(1; 3)$ & $(s,sr,r)$ & $(c_1, 2)$, \newline $(d_1, 2)$ \\
\hline
3.q.2 & $D_4$ & $(2^2,4^2)$ & $(s,s,r^{-1},r)$ & $(x_1x_2,1)$, $(x_1x_2x_3x_2^{-1},2)$, $(x_2x_3,2)$ \\
\hline
3.r.2 & $D_4$ & $(2^5)$ & $(s,s,sr, sr^3, r^2)$ & $(x_1x_2, 2)$, \newline $(x_3x_4,2)$, $(x_2x_3x_4, 2)$, $(x_3x_4x_1, 2)$, $(x_3x_1x_2, 2)$, $(x_1x_2x_3, 2)$ \\
\hline
3.s.2 & $D_4$ & $(1;2)$ & $(s,sr,r^2)$ & $(c_1, 2)$, \newline $(d_1, 2)$ \\
\hline
3.y & $D_6$ & $(2^3,6)$ & $(s,sr^2,r^3,r)$ & $(x_1x_2x_3x_2^{-1}, 1)$, $(x_2x_3, 1)$, \newline $(x_1x_2, 3)$ \\
\hline
3.z & $A_4$ & $(2^2,3^2)$ & $((12)(34),(12)(34),(123),(321))$ & $(x_1x_2, 1)$, $(x_1x_2x_3x_2^{-1}, 3)$, $(x_2x_3, 3)$ \\
\hline
3.al & $S_4$ & $(2^3,3)$ & $((12),(23),(13)(24),(243))$ & $(x_1x_2x_3x_2^{-1}, 1)$, $(x_2x_3, 2)$, \newline $(x_1x_2, 3)$ \\
\hline
\end{tabular}
\caption{A set of partial rotations generating an arithmetic subgroup in the 12 remaining cases.}
\label{table:results}
\end{table}

%-------------------------------------------------------------------------------------------------------------------------------------------------------------------

\newpage
\printbibliography

\end{document}